\documentclass[12pt]{amsart}

\usepackage{amsmath,amsthm,amsfonts,amssymb,mathrsfs}
\usepackage[all]{xy}

\RequirePackage[dvipsnames,usenames]{color}  

\usepackage[total={6.5in,8.75in}, top=1.2in, left=1.0in, includefoot]{geometry}
\usepackage{amssymb}
\usepackage{latexsym}
\usepackage{amsmath,amsthm}
\usepackage{aeguill}
\usepackage{amssymb}
\usepackage{mathrsfs}
\usepackage{hyperref}
\usepackage{etoolbox}
\usepackage{enumerate}
\usepackage[dvipsnames]{xcolor}
\usepackage{tikz}[2013/12/13]

\makeatletter
\pretocmd{\chapter}{\addtocontents{toc}{\protect\addvspace{15\p@}}}{}{}
\pretocmd{\section}{\addtocontents{toc}{\protect\addvspace{3\p@}}}{}{}
\makeatother

\makeatletter
\def\@tocline#1#2#3#4#5#6#7{\relax
  \ifnum #1>\c@tocdepth 
  \else
    \par \addpenalty\@secpenalty\addvspace{#2}%
    \begingroup \hyphenpenalty\@M
    \@ifempty{#4}{%
      \@tempdima\csname r@tocindent\number#1\endcsname\relax
    }{%
      \@tempdima#4\relax
    }%
    \parindent\z@ \leftskip#3\relax \advance\leftskip\@tempdima\relax
    \rightskip\@pnumwidth plus4em \parfillskip-\@pnumwidth
    #5\leavevmode\hskip-\@tempdima
      \ifcase #1
       \or\or \hskip .5em \or \hskip 1em \else \hskip 1.5em \fi%
      #6\nobreak\relax
    \dotfill\hbox to\@pnumwidth{\@tocpagenum{#7}}\par
    \nobreak
    \endgroup
  \fi}
\makeatother
\DeclareSymbolFont{bbold}{U}{bbold}{m}{n}
\DeclareSymbolFontAlphabet{\mathbbold}{bbold}

\newcommand{\UOne}{\mathbbold{1}}

\newcommand{\C}{\mathbb{C}}
\newcommand{\N}{\mathbb{N}}
\newcommand{\Z}{\mathbb{Z}}

\newcommand{\Q}{\mathbb{Q}}
\newcommand{\F}{\mathbb{F}}
\newcommand{\A}{\mathbb{A}}

\newcommand{\cH}{\mathcal{H}}

\newcommand{\Gal}{\operatorname{Gal}}

\newcommand{\ad}{\operatorname{ad}}
\newcommand{\ab}{\operatorname{ab}}

\renewcommand{\sc}{\operatorname{sc}}

\newcommand{\der}{\operatorname{der}}

\newcommand{\rank}{\operatorname{rank}}

\newcommand{\SO}{\mathrm{SO}}

\newcommand{\GL}{\mathrm{GL}}
\newcommand{\PGL}{\mathrm{PGL}}
\newcommand{\SL}{\mathrm{SL}}

\newcommand{\GO}{\mathrm{GO}}

\def\ab{\text{ab}}
\def\wab{\text{wab}}

\def\bG{\mathbf{G}}
\def\bH{\mathbf{H}}
\def\bS{\mathbf{S}}

\def\bT{\mathbf{T}}

\def\bY{\mathbf{Y}}

\def\bpx{\begin{pmatrix}}
\def\epx{\end{pmatrix}}

\DeclareMathOperator{\Frob}{Frob}
\DeclareMathOperator{\Syl}{Syl}
\newcommand{\regss}{\mathrm{reg.ss.}}
\newcommand{\ord}{\mathrm{ord}}

\newtheorem{thm}{Theorem}[section]
\newtheorem{theorem}[thm]{Theorem}
\newtheorem{cor}[thm]{Corollary}

\newtheorem{prop}[thm]{Proposition}
\newtheorem{lemma}[thm]{Lemma}

\newtheorem{conj}[thm]{Conjecture}

\newtheorem{remark}[thm]{Remark}
\newtheorem{lem}[thm]{Lemma}
\newtheorem{defi}[thm]{Definition}

\begin{document}

\title[]{Weak abelian direct summands and irreducibility of Galois representations}

\author{Gebhard B\"ockle}
\email{gebhard.boeckle@iwr.uni-heidelberg.de}
\address{
University of Heidelberg\\
IWR, Im Neuenheimer Feld 368,
69120 Heidelberg, Germany
}

\author{Chun-Yin Hui}
\email{chhui@maths.hku.hk}
\email{pslnfq@gmail.com}
\address{
Department of Mathematics\\
The University of Hong Kong\\
Pokfulam, Hong Kong
}


\begin{abstract}
Let $\rho_\ell$ be a semisimple $\ell$-adic representation of a number field $K$
that is unramified almost everywhere. We introduce a new notion called 
weak abelian direct summands of $\rho_\ell$ 
and completely characterize them, for example, if the algebraic monodromy of $\rho_\ell$ is connected.
If $\rho_\ell$ is in addition $E$-rational for some number field $E$, we prove 
that the weak abelian direct summands are locally algebraic (and thus de Rham).
We also show that the weak abelian parts of a connected semisimple Serre compatible system form again such a system.

Using our results on weak abelian direct summands, when $K$ is totally real and $\rho_\ell$ is the three-dimensional 
$\ell$-adic representation 
attached to a regular algebraic cuspidal automorphic, not necessarily polarizable representation $\pi$
of $\GL_3(\A_K)$ together with an isomorphism $\C\simeq \overline\Q_\ell$, 
we prove that $\rho_\ell$ is irreducible.
 We deduce in this case also some $\ell$-adic Hodge theoretic properties of $\rho_\ell$ if $\ell$ belongs 
to a Dirichlet density one set of primes.
\end{abstract}

\maketitle

\section{Introduction}
\subsection{Weak abelian direct summands}
Let $K$ be a number field,  
$\overline K$ an algebraic closure of $K$, and
$\Gal_K:=\Gal(\overline K/K)$ the absolute Galois group of $K$.
Denote by $\Sigma_K$ the set of finite places of $K$ and by $S_\ell$ the subset 
of places that divide $\ell$.
An $\ell$-adic representation of $K$
$$\rho_\ell:\Gal_K\to\GL_n(\overline\Q_\ell)$$
 is said to be \emph{$E$-rational} for some number field $E\subset\overline\Q_\ell$
if  for almost all $v\in\Sigma_K$, the Galois representation $\rho_\ell$ is unramified at $v$
and the characteristic polynomial of $\rho_\ell(\mathrm{Frob}_v)$ has coefficients in $E$.
Such $E$-rational $\ell$-adic representations provide Galois theoretic incarnations of 
many objects in number theory and arithmetic geometry.
For instance, Faltings isogeny theorem \cite{Fa83} 
asserts that two abelian varieties defined over $K$ are isogenous 
if their attached ($\Q$-rational) $\ell$-adic representations of $K$ are equivalent.
 
When $\rho_\ell$ is semisimple and abelian, $E$-rationality (for some $E$)
and \emph{local algebraicity} (see $\mathsection$\ref{s23} for definition) of $\rho_\ell$ are equivalent, and 
the latter condition is equivalent to the local abelian Galois representations 
$\rho_\ell|_{\Gal_{K_v}}$ being Hodge-Tate (or de Rham) for all $v\in S_\ell$. 
These facts are obtained by Serre's abelian $\ell$-adic representations theory \cite{Se98} together with 
a remarkable theorem of Waldschmidt on transcendental numbers \cite{Wa81}.
 For a general semisimple $E$-rational $\rho_\ell$
it is not known that the local representation $\rho_\ell|_{\Gal_{K_v}}$ is Hodge-Tate for $v\in S_\ell$.

Let $\psi_\ell$ be an abelian semisimple $\ell$-adic representation of $K$.
We say that $\psi_\ell$ is a \emph{weak abelian direct summand} of $\rho_\ell$ if 
there is a subset $\mathcal L\subset\Sigma_K$ of Dirichlet density one such that 
for all $v\in\mathcal L$,
$\psi_\ell$ and $\rho_\ell$ are unramified at $v$ and
the characteristic polynomial of $\psi_\ell( \mathrm{Frob}_v)$
divides those of $\rho_\ell(\mathrm{Frob}_v)$.
Abelian direct summands (i.e., subrepresentations) 
of $\rho_\ell$ are obvious examples of weak abelian direct summands;
if $n$ is odd and $\rho_\ell$ factors through $\GO_n=\SO_n\times\mathbb{G}_m$, then
the homothety character of $\rho_\ell$ is also a weak abelian direct summand.
The goals of this article are to investigate this new notion (more flexible than abelian subrepresentations) and give some applications
to compatible system and automorphic Galois representations.

Under mild conditions on the algebraic monodromy group $\bG_\ell$ of $\rho_\ell$ 
(e.g., when $\bG_\ell$ is connected), we completely characterize the
weak abelian direct summands of an $E$-rational semisimple $\rho_\ell$ (Proposition \ref{lem:WeakDiv2}).
This, together with Serre-Waldschmidt theory, lead to the main theorem below, which generalizes the implication
\begin{center}
$E$-rationality ~ $\Rightarrow$~ local algebraicity
\end{center}
for abelian semisimple $\rho_\ell$ to the general (i.e., non-abelian) situation.

\begin{thm}\label{abpart}
Let $K$ and $E\subset\overline\Q_\ell$ be number fields, and $\rho_\ell:\Gal_K\to\GL_n(\overline\Q_\ell)$
a semisimple $E$-rational $\ell$-adic representation of $K$. 
If $\psi_\ell$ is a weak abelian direct summand of $\rho_\ell$, then it is locally algebraic, 
and thus its local representations at places above $\ell$ are de Rham. 
\end{thm}

Let $E$ be a number field and $\{\rho_\lambda:\Gal_K\to\GL_n(\overline E_\lambda):~\lambda\in\Sigma_E\}$
a semisimple (Serre) compatible system of $K$ defined over $E$.
Denote the \emph{abelian part} (i.e., the maximal abelian subrepresentation) of $\rho_\lambda$
by $\rho_\lambda^{\ab}$. It is conjectured that  the family $\{\rho_\lambda^{\ab}\}$ is a compatible system;
but this seems a difficult question that is only resolved in some cases, cf.~\cite{Hu20}
Suppose the algebraic monodromy group $\bG_\lambda$ of $\rho_\lambda$ is connected for all $\lambda$.
The \emph{weak abelian part} of $\rho_\lambda$, denoted by $\rho_\lambda^{\wab}$, is defined to be the maximal weak abelian direct summand of $\rho_\lambda$. As an application of our results on 
weak abelian direct summands, we show
that $\{\rho_\lambda^{\wab}\}$ is likewise a compatible system (Theorem \ref{thm:WeakDiv}).
In next subsection, we describe our results on some three dimensional automorphic Galois representations.

\subsection{Automorphic Galois representations}
Let $K$ be a totally real (or CM) number field, $\pi$ a \emph{regular algebraic cuspidal automorphic 
representation} of $\GL_n(\A_K)$, and $\iota:\C\stackrel{\sim}{\rightarrow}\overline\Q_\ell$ 
a field isomorphism.
Denote by $S\subset\Sigma_K$ the subset of places $v$ such that the local component 
$\pi_v$ is ramified.
Attached to the pair $(\pi,\iota)$
is a semisimple (conjecturally irreducible, see e.g., \cite{Ram08}) 
$\ell$-adic representation of $K$
\begin{equation}\label{repn}
\rho_{\pi,\iota}:\Gal_K\to\GL_n(\overline\Q_\ell)
\end{equation}
that satisfies \emph{local-global compatibility away from $\ell$}: 
if $v$ is a finite place of $K$ outside $S_\ell$, then
\begin{equation}\label{lg}
\mathrm{WD}(\rho_{\pi,\iota}|_{\Gal_{K_v}})^{ss}=
\iota\circ \mathrm{rec}_{K_v}(\pi_v\otimes |\det|_v^{\frac{1-n}{2}}).
\end{equation}
The construction of $\rho_{\pi,\iota}$ such that \eqref{lg} holds for all $v$ outside $S\cup S_\ell$
was given by \cite{HLTT16} (and resp. \cite{Sc15}) 
and the validity of \eqref{lg} for all $v$ outside $S_\ell$ was obtained in \cite{Va18}.
If $\pi$ is in addition \emph{polarized}, the construction of such an $\ell$-adic representation
was done earlier and moreover, the following $\ell$-adic Hodge properties hold for all $v\in S_\ell$:
\begin{enumerate}
\item[(dR):] $\rho_{\pi,\iota}|_{\Gal_{K_v}}$ is de Rham with determined Hodge-Tate numbers;
\item[(Cr):] $\rho_{\pi,\iota}|_{\Gal_{K_v}}$ is crystalline if $\pi_v$ is unramified,
\end{enumerate}
see \cite[$\mathsection2.1$]{BLGGT14b} and the references therein for more details.
These properties are crucial to investigating the irreducibility (resp. residual irreducibility if $\ell\gg0$) 
of \eqref{repn} \cite{CG13},\cite{PT15},\cite{Xi19},\cite{Hu23a}  
by using potential automorphy theorems \cite{BLGGT14b} (resp. the method of algebraic envelopes \cite{Hu23a}).
For general $\pi$, it is not even known that \eqref{repn}
is Hodge-Tate at $v\in S_\ell$ unless more conditions (e.g., residual irreducibility of $\rho_{\pi,\iota}$) are assumed, see \cite{ACC+23},\cite{A'C23}. 

When $K$ is totally real and $n=3,4$, 
D. Ramakrishnan obtained  the irreducibility of \eqref{repn} 
assuming some $\ell$-adic Hodge properties \cite[Theorem B]{Ram13}.
The idea for $n=3$ is as follows. If on the contrary $\rho_{\pi,\iota}$ is reducible, then it contains 
a one-dimensional subrepresentation $\tau_\lambda$ (a weak abelian direct summand). 
Assuming $\tau_\lambda$ is Hodge-Tate, Ramakrishnan deduces a contradiction to 
the cuspidality of $\pi$
by an L-function argument.
As a consequence of Theorem \ref{abpart}, we obtain the following unconditional irreducibility result, which inspires this~article.


\begin{thm}\label{main}
Let $K$ be a totally real field, $\pi$ a
regular algebraic cuspidal automorphic 
representation of $\GL_3(\A_K)$, and $\iota:\C\stackrel{\sim}{\rightarrow}\overline\Q_\ell$
a field isomorphism. 
Then the three-dimensional $\ell$-adic Galois representation \eqref{repn} of $K$
attached to $(\pi,\iota)$
is irreducible.
\end{thm}

In the opposite direction, we deduce in Theorem \ref{mthm2} for $\rho_\ell$ of \emph{irreducible type} ($A_2$) that properties (dR) and (Cr) hold for rational primes $\ell$ in a Dirichlet density one set.

\subsection{Structure of the article}
We briefly describe the remaining sections of this article. 
We first recall basic notions and facts about (abelian) 
$\ell$-adic representations 
in $\mathsection2.1$--$\mathsection2.5$. 
Then we introduce weak abelian direct summands 
of an almost everywhere unramified semisimple $\ell$-adic Galois representation 
and completely characterize them under mild conditions on $\bG_\ell$ 
(Proposition \ref{lem:WeakDiv2}) 
in $\mathsection2.6$. The main Theorem \ref{abpart} is established  in $\mathsection2.7$ by 
using $\mathsection2.1$--$\mathsection2.6$ and Serre-Waldschmidt theory. 
Finally in $\mathsection2.8$--$\mathsection2.9$,
we formulate a conjecture on $E$-rationality of the abelian part $\rho_\ell^{\ab}$ (Conjecture \ref{E-rat-conj}) with some evidence (Theorem \ref{E-rat-evi})  and prove the compatibility of the weak abelian part of a connected semisimple  Serre compatible system (Theorem \ref{thm:WeakDiv}). 

In $\mathsection 3.1$, we introduce the Serre compatible system $\{\rho_{\pi,\lambda}\}$ 
of $\ell$-adic representations attached to $\pi$ and the algebraic monodromy groups $\bG_\lambda$.
When $K$ is totally real and $n=3$, we prove the irreducibility of $\rho_{\pi,\lambda}$ 
(Theorem \ref{main}) and $\lambda$-independence of $\bG_\lambda$ in $\mathsection 3.2$. Moreover, 
we deduce some $\ell$-adic Hodge properties of $\{\rho_{\pi,\lambda}\}$ if it is of type ($A_2$) 
(Theorem \ref{mthm2}) in $\mathsection 3.3$.

\section{Weak abelian direct summands of $\ell$-adic representation}
We fix some notation. 
Let $K$ be a number field and $\Sigma_K$ the set of finite places of $K$.
For a rational prime $\ell$, denote by $S_\ell$ 
the set of finite places $v\in\Sigma_K$ above $\ell$.
An \emph{$\ell$-adic representation} of $K$
is a continuous homomorphism $\rho_\ell:\Gal_K\to\GL_n(\overline\Q_\ell)$.
The \emph{algebraic monodromy group} $\bG_\ell$ of $\rho_\ell$
is the Zariski closure of the image $\rho_\ell(\Gal_K)$ in $\GL_{n,\overline\Q_\ell}$. 

\subsection{$E$-rationality}\label{s21}
Let $E\subset\overline\Q_\ell$ be a number field. An $n$-dimensional $\ell$-adic representation 
\begin{equation*}
\rho_\ell:\Gal_K\to\GL_n(\overline\Q_\ell)
\end{equation*}
is said to be \emph{$E$-rational} \cite[Chap. I]{Se98} if 
 there is a finite subset $S\subset\Sigma_K$ 
such that for every finite place $v\in\Sigma_K\backslash S$,
the representation $\rho_\ell$ is unramified at $v$ and the characteristic polynomial
of $\rho_\ell(\Frob_v)$ has coefficients in $E$, i.e.,
\begin{equation*}\label{charpoly}
\det(\rho_\ell(\Frob_v)-T\cdot\text{Id})\in E[T],
\end{equation*}
where $\Frob_v$ denotes the \emph{Frobenius class} at $v$.

\subsection{Serre compatible system}\label{s22}
Let  $S\subset\Sigma_K$ a finite subset and $E$ a number field.
For $\lambda\in\Sigma_E$, denote its residue characteristic by $\ell(\lambda)$ 
(and so that $\overline E_\lambda\cong\overline\Q_{\ell(\lambda)}$).
A family of $n$-dimensional $\ell$-adic representations (indexed by $\Sigma_E$)
\begin{equation}\label{csdef}
\{\rho_\lambda:\Gal_K\to \GL_n(\overline E_\lambda):~\lambda\in\Sigma_E\}
\end{equation}
is said to be a \emph{Serre compatible system} (SCS) of $K$ defined over $E$
 unramified outside $S$ (\cite[Chap. 1]{Se98}) if 
there exists $P_v(T)\in E[T]$ for all $v\in\Sigma_K\backslash S$ such that
for any $\lambda\in\Sigma_E$ and $v\in \Sigma_{K}\backslash (S\cup S_{\ell(\lambda)})$, 
the representation $\rho_\lambda$ is unramified at $v$ and 
$$\det(\rho_\lambda(\Frob_v)-T\cdot\text{Id})=P_v(T)\in E[T].$$
The compatible system \eqref{csdef} is said to be \emph{semisimple}
if $\rho_\lambda$ 
is semisimple for all $\lambda$.

\subsection{Local algebraicity}\label{s23}
Define $\bT_K:=\mathrm{Res}_{K/\Q}\mathbb{G}_{m,K}$ 
to be the $[K:\Q]$-dimensional $\Q$-torus given by restriction of scalars of 
$\mathbb{G}_{m,K}:=\mathrm{Spec}K[x,x^{-1}]$.
For any commutative $\Q$-algebra $A$,  its group of $A$-points is
$$\bT_K(A)=(K\otimes_\Q A)^\times.$$
Let $\A_K^\times$ be the group of id\`eles of $K$ and $i_\ell$ be the natural composition
\begin{equation*}\label{localArtin}
i_\ell:(K\otimes_\Q\Q_\ell)^\times=\prod_{v|\ell}K_v^\times\to \A_K^\times/K^\times \stackrel{\mathrm{Art}_K}{\longrightarrow}\Gal_K^{\text{ab}},
\end{equation*}
where $\mathrm{Art}_K$ is the Artin reciprocity map.

An abelian semisimple $\ell$-adic representation 
\begin{equation*}
\phi_\ell:\Gal_K^{\ab}\to\GL_n(\overline\Q_\ell)
\end{equation*}
is said to be \emph{locally algebraic} \cite[Chap. III]{Se98}
if there exists an $\overline\Q_\ell$-algebraic morphism
$$r_\ell:\bT_K\times_\Q\overline\Q_\ell \to \GL_{n,\overline\Q_\ell}$$
such that for all $x$ close enough to $1$ in the $\ell$-adic Lie group 
$(K\otimes_\Q\Q_\ell)^\times=\bT_K(\Q_\ell)$:
\begin{equation*}\label{eq:LocAlgMap}
\phi_\ell\circ i_\ell(x)=r_\ell(x^{-1}).
\end{equation*}

\begin{remark}
Local algebraicity of $\phi_\ell$ depends solely on the local representations
\begin{equation*}\label{locrepn}
\phi_\ell|_{\Gal_{K_v}}\hspace{.2in}\text{for all}~ v\in S_\ell.
\end{equation*}
By \cite[Appendix of Chap. III]{Se98}, $\phi_\ell$ is locally algebraic if and only if these local representations 
are Hodge-Tate. Moreover, being Hodge-Tate and being de Rham are equivalent
for these abelian local representations, see \cite[Proposition in $\mathsection$6]{FM95}.
\end{remark}

\subsection{Equivalence of three notions}\label{s24}
The following three notions of an abelian semisimple 
$$\phi_\ell:\Gal_K^{\ab}\to\GL_n(\overline\Q_\ell)$$ are equivalent.\\
	
	(Loc-alg): $\phi_\ell$ is locally algebraic.\\

 ($E$-SCS): $\phi_\ell\in \{\phi_\lambda\}_\lambda$ a semisimple abelian Serre compatible system
  of $K$ defined over some number field $E$ unramified outside some finite $S\subset\Sigma_K$.\\
	
	 ($E$-rat): $\phi_\ell$ is $E$-rational for some number field $E\subset\overline\Q_\ell$.\\

For direction (Loc-alg) $\Rightarrow$ ($E$-SCS), it is shown in \cite[Chap. III.2.3, proof of Theorem 2]{Se98}
that $\phi_\ell$ admits a factorization
\begin{equation}\label{factor}
\phi_\ell:\Gal_K^{\ab}\stackrel{\epsilon_\ell}{\longrightarrow}\bS_\mathfrak{m}(\overline\Q_\ell)\stackrel{\phi_0}{\longrightarrow}\GL_n(\overline\Q_\ell),
\end{equation}
where 
\begin{itemize}
\item $\bS_\mathfrak{m}$ is the Serre group (for some modulus $\mathfrak m$ of $K$)\footnote{The identity component $\bT_\mathfrak{m}$ of $\bS_\mathfrak{m}$ is a quotient of $\bT_K$
and the component group $\bS_\mathfrak{m}/\bT_\mathfrak{m}$ 
is isomorphic to the Ray class group of $\mathfrak{m}$.}, which is a diagonalizable group defined over $\Q$;
\item $\epsilon_\ell:\Gal_K^{\ab}\to \bS_\mathfrak{m}(\Q_\ell)$ is the canonical morphism 
constructed (for every $\ell$) in \cite[Chap. II.2.3]{Se98};
\item $\phi_0$ is an $\overline\Q_\ell$-algebraic morphism, which we assume to be a direct sum of characters.
\end{itemize}
Since $\bS_\mathfrak{m}$ is defined over $\Q$, we may assume $\phi_0$ to be defined over 
some number field $E$. Then the family 
\begin{equation}\label{Serresys}
\{\phi_\lambda:\Gal_K^{\ab}\stackrel{\epsilon_{\ell(\lambda)}}{\longrightarrow}\bS_\mathfrak{m}(E_\lambda)
\stackrel{\phi_0\times_E E_\lambda}{\longrightarrow}\GL_n(E_\lambda)\subset\GL_n(\overline E_\lambda):~\lambda\in\Sigma_E\}
\end{equation}
is the SCS we want.

  The direction ($E$-SCS) $\Rightarrow$ ($E$-rat) is obvious. Finally for 
  ($E$-rat) $\Rightarrow$ (Loc-alg), 
according to the strategy of Serre in \cite[Chap. III.3]{Se98} there are two steps:
the first one is to show that $\phi_\ell$ is \emph{almost locally algebraic} (i.e., $\phi_\ell^N$ is locally algebraic for some $N\in\N$) and the second one is to show that almost local algebraicity implies local algebraicity for an $E$-rational $\phi_\ell$ (see Proposition \ref{Serre1}).
The first step is a direct consequence of Theorem \ref{thm:Waldschmidt} below which 
is due to Waldschmidt \cite{Wa81} (see e.g., \cite[Theorem 7]{He82}).

\begin{theorem}[Waldschmidt]\label{thm:Waldschmidt}
Suppose $\tau_\ell:\Gal_K\to\overline\Q_\ell^\times$ is an one-dimensional $\ell$-adic representation unramified outside a finite set $S\subset \Sigma_K$ such that $\tau_\ell(\Frob_v)\in \overline\Q$ for all $v\in \Sigma_K\setminus S$. Then there exists $N\in\N$ such that $\tau_\ell^N$ is locally algebraic.
\end{theorem}

\subsection{
{Formal character}}
Let $F$ be an algebraically closed field 
and $\bG\subset\GL_{n,F}$ a reductive subgroup.
The \emph{formal character} of $\bG$ is defined as a subtorus 
\begin{equation}\label{fc}
\bT\subset\GL_{n,F}
\end{equation}
up to $\GL_n(F)$-conjugation
such that $g\bT g^{-1}$ is a maximal torus of $\bG$ for some $g\in\GL_n(F)$.

After some conjugation, \eqref{fc} can be diagonalized and thus admits some $\Z$-forms:
\begin{equation}\label{Zformfc}
\bT_\Z\subset\mathbb{G}_{m,\Z}^n.
\end{equation}
If $\bG'\subset\GL_{n,F'}$ is another reductive subgroup ($F'$ some algebraically closed field),
we say that the formal characters of $\bG$ and $\bG'$ are \emph{the same} if 
they admit the same $\Z$-form \eqref{Zformfc} after some diagonalizations.

\subsection{Weak abelian direct summands} 
{Throughout this subsection, 
we denote by $F$ the algebraically closed field $\overline\Q_\ell$ for simplicity.
Let $\rho:\Gal_K\to\GL_n(F)$ be a semisimple $\ell$-adic representation that is unramified almost everywhere.
We investigate a notion weaker than (abelian) subrepresentation of $\rho$, 
called \emph{weak (abelian) direct summand}. When the algebraic monodromy group $\bG_\rho$
satisfies some hypotheses (e.g., $\bG_\rho$ is connected), weak abelian direct summands of $\rho$
are completely characterized and the \emph{weak abelian part} (i.e., the maximal weak abelian direct summand) 
of $\rho$ is well-defined (see Proposition \ref{lem:WeakDiv2}).}

\subsubsection{
{Weak divisibility}}

\begin{defi}\label{def:WeakDiv}
Let $\rho:\Gal_K\to\GL_n(F)$ and $\psi:\Gal_K\to\GL_m(F)$ be semisimple $\ell$-adic 
representations that are unramified almost everywhere. Denote by
 $S_\rho,S_\psi\subset \Sigma_K$ the set of ramified places of $\rho$ and $\chi$, respectively. 
We say that $\psi$ is a weak direct summand of $\rho$ (or $\psi$ weakly divides $\rho$) if 
the set
\begin{equation}\label{def:Places-psi-wd-rho}
S_{\psi\mid\rho}:=\{v\in\Sigma_K\setminus (S_\rho\cup S_\psi): \det(\psi(\Frob_v)-T\cdot\text{Id}) ~\mathrm{\,divides~\,}\det(\rho(\Frob_v)-T\cdot\text{Id})\} 
\end{equation}
is of Dirichlet density one. If $\psi$ is abelian and weakly divides $\rho$, 
we say that $\psi$ is a weak abelian direct summand of $\rho$.
\end{defi}

{Similarly, one can define weak divisibility for a pair of semisimple Serre compatible systems of $K$ defined over $E$. Note that if $\psi$ is a subrepresentation of $\rho$ then $\psi$ weakly divides $\rho$.
}

\begin{prop} \label{rem:WeakDiv}
If $\psi$ weakly divides $\rho$, then $S_{\psi\mid\rho}=\Sigma_K\setminus (S_\rho\cup S_\psi)$.
\end{prop}

\begin{proof}
Let $n$ and $m$ be the degrees of $\rho$ and $\psi$. For $\sigma\in \Gal_K$, denote by $\Syl_{\lambda,\sigma}$ the Sylvester matrix in $M_{(n+m)\times(n+m)}(E_\lambda)$ for the polynomials 
$r_{\lambda,\sigma}(T):=\det(\rho(\sigma)-T\cdot\text{Id})$ and 
$s_{\lambda,\sigma}(T):=\det(\psi(\sigma)-T\cdot\text{Id})$ in $F[T]$ of degrees $n$ and $m$, respectively. By elementary linear algebra one has the following well-known result
\begin{equation}
\rank \Syl_{\lambda,\sigma} = n+m - \deg \gcd(r_{\lambda,\sigma},s_{\lambda,\sigma}).
\end{equation}
Fix $r$ row and column labels in $\{1,\ldots,n+m\}$, i.e., $r$ element subsets $I,J\subset \{1,\ldots,n+m\}$, and consider the determinant of the $r\times r$ minor of $\Syl_{\lambda,\sigma}$ spanned by the chosen rows and columns; call it $d_{\lambda,I,J}(\sigma)$. Then for each $\lambda,I,J$ the map $d_{\lambda,I,J}:G_K\to E_\lambda$ that sends $\sigma\mapsto d_{\lambda,I,J}(\sigma)$ 
is continuous, and by our hypothesis, whenever $|I|=|J|>n$, the map vanishes on the elements $\Frob_v$, for $v$ in a density one subset of places of $K$. Hence by the Chebotarev density theorem the map  $ d_{\lambda,I,J}$ is identically zero for any $|I|=|J|>n$. It follows that $s_{\lambda,\sigma}$ divides $r_{\lambda,\sigma}$ for all $\sigma\in \Gal_K$, and choosing $\sigma=\Frob_v$ for $v\in \Sigma_K\setminus (S_\rho\cup S_\psi)$ and $\lambda$ a place of $E$ with residue characteristic different from $v$, the result follows.
\end{proof}

\begin{remark}
{\em The notion of weak divisibility from Definition~\ref{def:WeakDiv} has some pitfalls that we'd like to point out: It is easy to see that if $\psi$ weakly divides $\rho$ then the representation $\psi$ need not be a direct summand of $\rho$; for instance, the trivial $1$-dimensional representation is a weak direct summand of any $3$-dimensional $\ell$-adic representation with image in $\SO_3$. But in fact one can have $\rho$ and $\psi$ abelian and semisimple with $\psi$ weakly dividing $\rho$, and yet $\psi$ not being a direct summand of $\rho$. Let $L/K$ be a finite Galois extension with $\Gal(L/K)\cong\Z/2\times\Z/2$, let $\chi_i$, $i=0,1,2,3$, denote the characters $\Gal(L/K)\to\{\pm1\}$. Then any $\chi_i$ is a weak direct summand of the sum of the other $3$ characters.

In the abelian case, this type of problem can arise whenever there are characters $\chi$ and $\chi'$ that are direct summands of $\psi$ and $\rho$ such that $\chi'\otimes\chi^{-1}$ has finite non-trivial order. Ruling out this case, below in Proposition \ref{lem:WeakDiv2} and Corollary \ref{cor:WeakDiv2} we shall analyze in detail when an abelian $\psi$ weakly divides $\rho$.}
\end{remark}

\subsubsection{
{Some lemmas}}
For the next results we denote by $\mathbb{G}_m$ the center of $\GL_{n,F}$, and by $\UOne:\Gal_K\to\GL_1(F)$ the trivial character.
For a reductive group $\bG$ defined over $F$, denote by $\bG^{\der}$ the derived group of the identity component $\bG^\circ$.

\begin{lem}\label{lem:OnNuG}
Let $\bG\subset\GL_{n,F}$ be a closed subgroup that acts irreducible  on $F^n$ 
and suppose that $\bG/Z(\bG)$ is connected, where $Z(\bG)\subset\bG$ is the center. 
Let $n_0$ be the multiplicity of the weight zero in the formal character of
$\bG^{\der}$ on $F^n$. 
Then the following hold:
\begin{enumerate}
\item[(i)] The center $Z(\bG)$ lies in $\mathbb{G}_m$, and so is equal to $\bG\cap\mathbb{G}_m$, an element $\zeta\cdot1$ in the center acts by multiplication with $\zeta$ on $F^n$. Moreover, if $Z(\bG)$ is finite then $\bG^\circ$ is semisimple. 
In all cases we have $\bG=(\bG\cap\mathbb{G}_m)\bG^{\der}$.
\item[(ii)] Suppose $n_0>0$. Then we have:
 $\bG^{\der} \cap \mathbb{G}_m=\{1\}$ and a direct product
\[ \bG=\bG^{\der}\times (\bG\cap\mathbb{G}_m).\]
Moreover denoting by $\nu_{\bG}:\bG\to \bG\cap\mathbb{G}_m$ the projection onto the second factor,  one has $\ker(\nu_\bG)=\bG^{\der}$ and $\det=(\nu_\bG)^{\otimes n}$ with $\det$ the map $\bG\subset\GL_n\stackrel\det\to\mathbb{G}_m$.
\end{enumerate}
\end{lem}

\begin{proof}
The first sentence of (i) follows from the irreducibility of the action of $\bG$ on $F^n$. The connectedness of $\bG/Z(\bG)$ implies that $\bG= Z(\bG) \bG^\circ$. It follows that $\bG^\circ$ acts irreducibly and is thus reductive connected. In particular $Z(\bG^\circ)$ lies in $\mathbb{G}_m$ 
and one has $Z(\bG^\circ)\bG^{\der}=\bG^\circ$. Therefore, if $\bG\cap \mathbb{G}_m$ is finite, so must be $\bG^\circ\cap \mathbb{G}_m=Z(\bG^\circ)$ and then $\bG^\circ$ is semisimple and $\bG=(\bG^\circ\cap \mathbb{G}_m) \bG^{\der}$. Otherwise, $\bG$ contains $\mathbb{G}_m$ and 
we have $\bG=\bG^\circ=\mathbb{G}_m \bG^{\der}$.

For the proof of (ii) assume $n_0>0$. Since the intersection $\bG^{\der}\cap\mathbb{G}_m$ cannot contain any non-trivial scalar matrices, it follows that $\bG^{\der}\cap\mathbb{G}_m=\{1\}$. By (i) the center $Z(\bG)$ together with $\bG^{\der}$ generate $\bG$, and they are both normal subgroups of $\bG$. As we have shown, they intersect trivially, and now part (ii) follows from elementary group theory. The last assertions, on $\nu_{\bG}$, are then straightforward.
\end{proof}

\begin{lem}\label{lemtwist}
Let $\rho\colon \Gal_K\to \GL_n(F)$ and $\chi:\Gal_K\to \GL_1(F)$ be semisimple  $\ell$-adic representations.
Denote by $\bG_\rho$ and $\bG_{\rho\otimes\chi}$ respectively the algebraic monodromy groups of $\rho$ 
and $\rho\otimes\chi$. Then the following assertions hold.
\begin{enumerate}[(i)]
\item The quotient $\bG_\rho/Z(\bG_\rho)$ is connected if and only if 
$\bG_{\rho\otimes\chi}/Z(\bG_{\rho\otimes\chi})$ is connected.
\item The multiplicities of
the weight zero in the formal characters of $\bG_\rho^{\der}$ and $\bG_{\rho\otimes\chi}^{\der}$ on $F^n$ are equal.
\end{enumerate}
\end{lem}

\begin{proof}
The lemma follows from the facts that $\bG_\rho/Z(\bG_\rho)\simeq \bG_{\rho\otimes\chi}/Z(\bG_{\rho\otimes\chi})$
and the actions of $\bG_\rho^{\der}$ and $\bG_{\rho\otimes\chi}^{\der}$ on $F^n$ can be identified.
\end{proof}

To proceed, we recall the following result of Steinberg,  e.g.~\cite[Thm.~11.4.18]{Goodman-Wallach}, that implicitly made its appearance already in the proof of Proposition~\ref{den1} when analyzing the set $\bY_\ell$. If $\bG$ is connected reductive group over $F$, then the conjugation-invariant
\begin{equation}\label{eq:steinberg}
\bG^\regss:=\{g\in \bG\mid g\hbox{ is regular semisimple}\}\hbox{ is a dense open subset of }\bG.
\end{equation}
For a subset $T\subset \Sigma_K$, we write $\delta(T)$ for its (Dirichlet) density if it exists.

\begin{lem}\label{lem:WeakDiv}
Let $\rho\colon \Gal_K\to \GL_n(F)$ and $\chi:\Gal_K\to \GL_1(F)$ be $\ell$-adic representations that are both unramified outside a finite subset $S\subset \Sigma_K$. Denote by $\bG_\rho\subset\GL_{n,F}$  the algebraic monodromy group  of $\rho$ and by $n_0$ the multiplicity of the weight zero in the formal character of 
$\bG_\rho^{\der}$ on $F^n$. Suppose that $\rho$ is irreducible and $\bG_\rho/Z(\bG_\rho)$ is connected. If $n_0>0$, define the character $\nu_{\bG_\rho}:\Gal_K\to (\bG_\rho\cap \mathbb{G}_m)(F^\times)\subset F^\times$.
Then the following assertions hold.
\begin{enumerate}
\item[(i)]  If $n_0=0$, then $\delta(S_{\chi|\rho})=0$.
\item[(ii)] If $n_0>0$ and the order of $\det(\rho\otimes\chi^{-1})$ is infinite, then 
$\mathbb{G}_m\subset \bG_{\rho\otimes\chi^{-1}}$ and $\delta(S_{\chi|\rho})=0$.
\item[(iii)] If $n_0>0$ and the order of $\det(\rho\otimes\chi^{-1})$ is finite, then 
$Z(\bG_{\rho\otimes\chi^{-1}})=\bG_{\rho\otimes\chi^{-1}}\cap\mathbb{G}_m$ is finite, 
the symmetric difference of $S_{\chi|\rho}$ and 
$S_{\UOne|\nu_{\bG_{\rho\otimes\chi^{-1}}}}$ is of density zero and 
 \[\delta(S_{\chi|\rho})=\delta(S_{\UOne|\nu_{\bG_{\rho\otimes\chi^{-1}}}})
=\frac1{|Z(\bG_{\rho\otimes\chi^{-1}})|}.\]
\item[(iv)] The character $\chi$ weakly divides $\rho$ if and only if $n_0>0$ 
and $\chi=\nu_{\bG_\rho}$. 
\item[(v)] If $\chi$ weakly divides $\rho$, then $\delta(S_{\chi^{\oplus n_0}|\rho})=1$ and $\delta(S_{\chi^{\oplus (n_0+1)}|\rho})=0$. 
\end{enumerate}
\end{lem}

\begin{proof}
Since the symmetric difference of $S_{\chi|\rho}$ and $S_{\UOne|\rho\otimes\chi^{-1}}$ is finite, 
it follows that
$\delta(S_{\chi|\rho})=\delta(S_{\UOne|\rho\otimes\chi^{-1}})$. As all our assertions concern $\rho\otimes\chi^{-1}$, we may assume $\chi=\UOne$ by Lemma \ref{lemtwist}. Note also that $\bG=\bG_\rho$ satisfies the hypotheses of Lemma~\ref{lem:OnNuG}, so that the lemma applies.

Suppose first that $\bG_\rho$ contains $\mathbb{G}_m$, 
which happens exactly when $\det(\rho)$ has infinite order. 
Then $\bG_\rho$ is connected reductive and has center $\mathbb{G}_m$. 
The  conjugation-invariant subvariety 
$$Y_1:=\{g\in \bG : 1\hbox{ is an eigenvalue of }g\}\subset \bG_\rho$$ 
is closed and no $g\in Y_1$ is regular semisimple because $\mathbb{G}_m$ does not lie in the closure of $g^\Z$. Hence, \eqref{eq:steinberg} implies that $Y_1$ is properly contained in $\bG_\rho$. 
Because $\bG_\rho$ is connected, it follows that $S_1:=\{v\in\Sigma_K\backslash S: \rho(\Frob_v)\in Y_1\}$ has density zero, and hence $\delta(S_{\UOne|\rho})=0$. Thus, we obtain (i) in case $\mathbb{G}_m\subset \bG_\rho$ and also (ii).

Suppose now that $\bG_\rho$ does not contain $\mathbb{G}_m$. 
Then $\bG_\rho^\circ=\bG_\rho^{\der}$ 
and $\bG_\rho\cap\mathbb{G}_m=Z(\bG_\rho)$ is a finite subgroup 
$\mu_N=\{\zeta\in F: \zeta^N=1\}$ of $\mathbb{G}_m$, 
where $N=|Z(\bG_\rho)|\in\N$.
Because $\bG_\rho=Z(\bG_\rho)\bG_\rho^{\der}$, any element in $\bG_\rho$ can be written as $\zeta g$ with $\zeta\in\mu_N$ and $g\in \bG_\rho^{\der}$. Now from \eqref{eq:steinberg} it follows that 
$$\{g\in \bG_\rho^{\der}: \hbox{ some eigenvalue of }g~\hbox{belongs to } \mu_N\backslash\{1\} \}$$ 
is a conjugation-invariant proper closed subvariety of $\bG_\rho^{\der}$, and therefore 
\begin{equation}
Y_2:=\{g\in \bG_\rho\setminus\bG_\rho^\circ: 1\hbox{ is an eigenvalue of }g\}
\end{equation}
is a conjugation-invariant closed 
nowhere dense subvariety of $\bG_\rho\setminus \bG^\circ_\rho$. 
Hence, the Chebotarev density theorem implies that
the set $S_2:=\{v\in\Sigma_K\backslash S: \rho(\Frob_v)\in Y_2\}$ has density zero. 
By the definition of $n_0$, 
the conjugation-invariant
\begin{equation}
Y_3:=\{g\in \bG_\rho^\circ: 1\hbox{ is an eigenvalue of }g\}\subset\bG_\rho^\circ
\end{equation}
is all of $\bG_\rho^\circ$ if $n_0>0$ and is proper closed subvariety if $n_0=0$. 
Hence, the Chebotarev density theorem implies that 
the set $S_3:=\{v\in\Sigma_K\backslash S: \rho(\Frob_v)\in Y_3\}$ has density $1/[\bG_\rho:\bG^\circ_\rho]$ 
if $n_0>0$ and density zero if $n_0=0$. 
In case $n_0=0$, we obtain that $\delta(S_{\UOne|\rho})=\delta(S_2)+\delta(S_3)=0$, and this completes~(i). 
In case $n_0>0$, Lemma~\ref{lem:OnNuG} yields 
$$[\bG_\rho:\bG^0_\rho]=|\mathbb{G}_m\cap\bG_\rho|=|Z(\bG_\rho)|=N.$$ 
By the above results on $\delta(S_2)$ and $\delta(S_3)$, we deduce (iii), 
and in particular
$$\delta(S_{\UOne|\rho})=\delta(S_2)+\delta(S_3)=\frac1N=\delta(S_{\UOne|\nu_{\bG_\rho}}).$$ 

For (iv), it suffices to handle the ``only if'' part. 
If $\UOne$ weakly divides $\rho$, it follows that 
$\delta(S_{\UOne|\rho})=1$ and we are in the situation of assertion (iii).
Hence, we obtain $\delta(S_{\UOne|\nu_{\bG_\rho}})=1$ and thus $\UOne=\nu_{\bG_\rho}$ 
by the Chebotarev density theorem.
The first part of (v) is clear from the definition of $n_0$. To see the second, note that under the hypotheses of (v) the group $\bG_\rho$ is semisimple. By the definition of $n_0$ the conjugation-invariant closed subset 
\begin{equation}
Y_4:=\{g\in \bG_\rho: 1\hbox{ is an eigenvalue of }g\hbox{ of multiplicity at least $n_0+1$} \}
\end{equation}
contains no regular semisimple elements, and by \eqref{eq:steinberg} it is nowhere dense in $\bG_\rho$. 
By the Chebotarev density theorem, we deduce that $\{v\in\Sigma_K\backslash S: \rho(\Frob_v)\in Y_4\}=S_{\UOne^{\oplus(n_0+1)}|\rho}$ has density zero.
\end{proof}

\subsubsection{
{Characterization of weak abelian direct summands and the weak abelian part}}

\begin{prop}\label{lem:WeakDiv2}
Let $\rho\colon \Gal_K\to \GL_n(F)$ be a  semisimple $\ell$-adic representation unramified outside a finite subset $S\subset \Sigma_K$ that can be written as  $\rho=\bigoplus_{j\in J}\rho_j$ with absolutely irreducible representations $\rho_j:\Gal_K\to\GL_{r_j}(F)$ and suppose that for each $j\in J$ the group $\bG_{\rho_j}/Z(\bG_{\rho_j})$ is connected, where $\bG_{\rho_j}$ is the algebraic monodromy group of $\rho_j$. 
Denote by $n_{j,0}$ the multiplicity of the weight zero in the formal character of 
$\bG_{\rho_j}^{\der}$ on $F^{r_j}$, let $J_0:=\{j\in J:n_{j,0}>0\}$ and for $j\in J_0$ define 
$\xi_j:=\nu_{\bG_{\rho_j}}$ the Galois character in Lemma \ref{lem:WeakDiv}. 
Suppose for all $j,j'\in J_0$ that $\xi_{j'}\xi_j^{-1}$ is either trivial or has infinite order. 
Define a partition $J_0=\bigsqcup_{a\in A}J_a$ be requiring $j,j'$ 
to lie in the same class if and only if $\xi_j=\xi_{j'}$, 
and write $\xi_a$ for $\xi_j$ if $j\in J_a$. 

Let also $\psi=\bigoplus_{i\in I}\psi_i^{\oplus e_i}$ be an abelian $\ell$-adic representation with continuous characters $\psi_i:\Gal_K\to F^\times$ unramified outside $S$, that are pairwise distinct and with $e_i\in\N$.
 Then $\psi$ weakly divides $\rho$ if and only if the following conditions hold:
\begin{enumerate}[(a)]
\item for each $i\in I$ there is an (unique) $a_i\in A$ such that $\psi_i=\xi_{a_i}$, and
\item for all $i\in I$ one has $e_i\le \sum_{j\in J_{a_i}} n_{j_i,0}$. 
\end{enumerate}
In particular, there is a unique 
degree $n_0:=\sum_{j\in J} n_{j,0}$ weak abelian direct summand $\rho^{\wab}$ of $\rho$
such that any weak abelian direct summand $\psi$ of $\rho$ is a subrepresentation of $\rho^{\wab}$.
We call $\rho^{\wab}$ the weak abelian part of $\rho$.

{Moreover if $\psi$ weakly divides $\rho$ and 
$V_\psi$ is the representation space of $\psi$, then 
there is an algebraic representation $r_\psi:\bG_\rho\to \GL_{V_\psi}$ 
such that $\psi$ is the composition 
\begin{equation}\label{algfactor}
\Gal_K\stackrel{\rho}{\rightarrow}\rho(\Gal_K)\hookrightarrow\bG_\rho(F)
\stackrel{r_\psi}{\rightarrow}\GL_{V_\psi}(F)=\GL(V_\psi).
\end{equation}
In particular, $r_\psi(\bG_\rho)=\bG_\psi$ is the algebraic monodromy group of $\psi$.}
\end{prop}

\begin{proof}
Suppose first that (a) and (b) hold. Then from Lemma~\ref{lem:WeakDiv}, it follows that $\psi_i^{\oplus e_i}$ weakly divides $\bigoplus_{j\in J_{a_i}} \rho_{j}$. Summing over all $i$ gives that $\psi$ weakly divides~$\rho$.

Suppose conversely that $\psi$ weakly divides $\rho$. Then in particular for each $i$ we have $\delta(S_{\psi_i|\rho})=1$. If $j\in J\setminus J_0$, then $\delta(S_{\psi_i|\rho_j})$ is zero by Lemma~\ref{lem:WeakDiv}(i), because $n_{j,0}=0$. If $j\in J_0$, then $\xi_j$ is defined. By our hypothesis on the $\xi_j$, the order of $\xi_a\otimes\psi_i^{-1}$ can be finite for at most one $a\in A$. If the order is infinite, by Lemma~\ref{lem:WeakDiv}(ii) we have $\delta(S_{\psi_i|\rho_j})=0$. Therefore there must be (exactly) one $a\in A$ such that $\delta(S_{\psi_i|\bigoplus_{j\in J_a}\rho_j})=1$. From  Lemma~\ref{lem:WeakDiv}(iii) we deduce that the symmetric difference of $S_{\psi_i|\rho_j}$ and $S_{\psi_i|\xi_a}$ is of density zero and that $\delta(S_{\psi_i|\xi_a})=\frac1{\ord(\psi_i\xi_a^{-1})}$ for $j\in J_a$. It follows that 
$$\frac1{\ord(\psi_i\xi_a^{-1})}=\delta(S_{\psi_i|\xi_a})
=\delta(\bigcup_{j\in J_a}S_{\psi_i|\rho_j})=\delta(S_{\psi_i|\bigoplus_{j\in J_a}\rho_j})=1$$
and we must have $\psi_i=\xi_{a}$ for an (unique) $a\in A$, which we now denote by $a_i$. 

It remains to prove (b). From the above we have $\delta(S_{\psi_i|\bigoplus_{j\in J_{a_i}}\rho_j})=1$. We know that
\[  S_{\psi_i^{\oplus e_i}|\bigoplus_{j\in J_{a_i}}\rho_j} = \bigcup_{(d_j)_{j\in J_{a_i}}} \bigcap_{j\in J_{a_i}} S_{\psi_i^{\oplus d_j}|\rho_j}  \]
where the union is over all $J_{a_i}$-indexed tuples $(d_j)$ of elements in $\Z_{\geq0}$ such that $\sum_{j\in J_{a_i}} d_j=e_i$. If $d_j>n_{j,0}$, then $\delta(S_{\psi_i^{\oplus d_j}|\rho_j}  )=0$ by  Lemma~\ref{lem:WeakDiv}(v); otherwise the density is one. Hence there must be a tuple $(d_j)$ such that $\delta(S_{\psi_i^{\oplus d_j}|\rho_j}  )=1$ for all $j\in J_{a_i}$. But then $d_j\le n_{j,0}$, and so 
\begin{equation}\label{wabcond}
e_i=\sum_{j\in J_{a_i}} d_j\le \sum_{j\in J_{a_i}}n_{j,0},
\end{equation} as was to be shown. 
When the equality of \eqref{wabcond} holds for all $i$, one obtains the weak abelian part $\rho^{\wab}=\psi$ of $\rho$ with the desired properties.

{Finally note that if $\psi$ weakly divides $\rho$,
we obtain $r_\psi$ and \eqref{algfactor}
by condition (a) and (the direct product in) Lemma \ref{lem:OnNuG}(ii). Since $\rho(\Gal_K)$
is Zariski dense in $\bG_\rho$, \eqref{algfactor} implies that 
$\psi(\Gal_K)$ is Zariski dense in $r_\psi(\bG_\rho)$ and thus
$\bG_\psi=r_\psi(\bG_\rho)$.}
\end{proof}

\begin{cor}\label{cor:WeakDiv2}
Let $\rho\colon \Gal_K\to \GL_n(F)$ and $\psi:\Gal_K\to\GL_m(F)$ be semisimple $\ell$-adic representations unramified outside a finite subset $S\subset \Sigma_K$, and suppose that $G_\rho$ is connected. Then the following assertions hold.
\begin{enumerate}[(i)]
\item The hypotheses of Proposition~\ref{lem:WeakDiv2} on $\rho$ are satisfied
under the connectedness of $\bG_\rho$.
\item If $\psi$ is abelian and weakly divides $\rho$, 
then the algebraic monodromy $\bG_\psi$ of $\psi$ is connected.
\item If $\rho$ and $\psi$ are abelian, then $\psi$ weakly divides $\rho$ if and only if $\psi$ is a direct summand of $\rho$.
\end{enumerate}
\end{cor}

\begin{proof} 
Since $\bG_\rho$ is connected, 
it is obvious that $\bG_{\rho_j}/Z(\bG_{\rho_j})$ is connected for all $j\in J$. 
Suppose $\xi_j$ and $\xi_{j'}$
are two Galois characters in Proposition \ref{lem:WeakDiv2} where $j,j'\in J_0$,
and let $\bG_{jj'}\subset\mathbb{G}_m$ be the algebraic monodromy
of $\xi_j \xi_{j'}^{-1}$.
By Lemma \ref{lem:OnNuG}(ii), there is 
an algebraic representation $r_{jj'}:\bG_\rho\to \mathbb{G}_m$
such that $\xi_j \xi_{j'}^{-1}$ is the composition 
\begin{equation}\label{algfactor1}
\Gal_K\stackrel{\rho}{\rightarrow}\rho(\Gal_K)\hookrightarrow\bG_\rho(F)
\stackrel{r_{jj'}}{\rightarrow}\mathbb{G}_m(F)= F^\times.
\end{equation}
Since $\rho(\Gal_K)$
is Zariski dense in $\bG_\rho$, \eqref{algfactor1} implies that 
$\xi_j \xi_{j'}^{-1}(\Gal_K)$ is Zariski dense in $r_{jj'}(\bG_\rho)$,
which is $\bG_{jj'}$. It follows that $\bG_{jj'}$ is also connected and thus, 
$\xi_j \xi_{j'}^{-1}$ is either trivial or infinite. 
We obtain (i).

By (i) and (the last assertion of)
Proposition \ref{lem:WeakDiv2}, we obtain a surjection $\bG_\rho\to\bG_\psi$ of algebraic groups. 
Since $\bG_\rho$ is connected, $\bG_\psi$ is also connected, which is (ii).
Finally, (iii) follows directly from (i) and Proposition \ref{lem:WeakDiv2} since $\rho$ is 
equal to its weak abelian part $\rho^{\wab}$.
\end{proof}

\subsection{
{Local algebraicity of weak abelian direct summand}}\label{s25}
Let $\rho_\ell:\Gal_K\to\GL_n(\overline\Q_\ell)$ be an $E$-rational 
semisimple $\ell$-adic representation for some number field $E\subset\overline\Q_\ell$
{and let $\psi_\ell$ be a weak abelian direct summand of $\rho_\ell$}.
Since $\rho_\ell$ is semisimple, the algebraic monodromy group 
$\bG_\ell$ of $\rho_\ell$ is a reductive subgroup of $\GL_{n,\overline\Q_\ell}$.
We are now ready to prove that $\psi_\ell$ is locally algebraic (Theorem \ref{abpart}).\\

\noindent \textbf{Proof of Theorem \ref{abpart}.}
For any finite field extension $L/K$, the restriction $\rho_\ell|_{\Gal_L}$
is also $E$-rational. Since the local algebraicity of $\psi_\ell$
follows from the local algebraicity of $\psi_\ell|_{\Gal_L}$,
we may assume that the algebraic monodromy group $\bG_\ell$ is connected
by replacing $K$ by some finite extension. 
{Since $\psi_\ell$ weakly divides $\rho_\ell$, 
there exists $N\in\N$ such that $\psi_\ell^N$ is locally algebraic 
by $E$-rationality of $\rho_\ell$ and Theorem \ref{thm:Waldschmidt}. }
Then the  implication (Loc-alg) $\Rightarrow$ ($E$-rat) for semisimple abelian
$\ell$-adic representations implies the existence of a number field $E'\subset\overline\Q_\ell$ 
such that $\psi_\ell^N$ is $E'$-rational.
By Propositions \ref{den1} and \ref{Serre1} below, $\psi_\ell$ is locally algebraic.
\qed

\begin{prop}\label{den1}
Let $\rho_\ell:\Gal_K\to\GL_n(\overline\Q_\ell)$ be a semisimple $E'$-rational $\ell$-adic representation 
unramified outside a finite $S\subset\Sigma_K$
such that the algebraic monodromy group $\bG_\ell$ is connected. 
{If $\psi_\ell$ is a weak abelian direct summand of $\rho_\ell$ and} there exists
$N\in\N$ such that $\psi_\ell^N$ is $E'$-rational, 
then there is a Dirichlet density one subset $\mathcal L_K\subset\Sigma_K\backslash S$
such that 
\begin{equation}\label{char}
\det(\psi_\ell(\Frob_v)-T\cdot \mathrm{Id})\in  E'[T]\hspace{.2in}\text{for all}~ v\in \mathcal L_K.
\end{equation}
\end{prop}

\begin{proof}
As $\psi_\ell$ is abelian, write
\begin{equation}\label{rhoD}
\psi_\ell=\chi_1\oplus\chi_2\oplus\cdots\oplus\chi_m
\end{equation}
as a direct sum of characters.
Since $\psi_\ell^N$ is $E'$-rational, 
for all $v$ outside some finite $S'\subset\Sigma_K$ the polynomial
\begin{equation}\label{product}
\prod_{i=1}^m(T^N-\chi_i(\Frob_v)^N)\in E'[T].
\end{equation}

{Let $V_{\rho_\ell}$ and $V_{\psi_\ell^{-1}}$ be the representation spaces of $\rho_\ell$ and $\psi_\ell^{-1}$ respectively. Denote by $i_\ell$ the faithful representation of $\bG_\ell$ on $V_{\rho_\ell}$.
Since $\bG_\ell$ is connected, the hypotheses of Proposition \ref{lem:WeakDiv2} on 
$\rho_\ell$ are satisfied by Corollary \ref{cor:WeakDiv2}(i).
Thus, Proposition \ref{lem:WeakDiv2} asserts that 
the twisted $\ell$-adic representation $\rho_\ell\otimes\psi_\ell^{-1}$
is equal to the composition
$$\Gal_K\stackrel{\rho_\ell}{\rightarrow}\rho_\ell(\Gal_K)\hookrightarrow\bG_\ell(\overline\Q_\ell)
\stackrel{i_\ell\otimes r_{\psi_\ell}^{-1}}{\longrightarrow}\GL_{V_{\rho_\ell}\otimes V_{\psi_\ell^{-1}}}(\overline\Q_\ell)
=\GL(V_{\rho_\ell}\otimes V_{\psi_\ell^{-1}}),$$
where $r_{\psi_\ell}$ is  as in \eqref{algfactor}.
Hence, the algebraic monodromy group 
$\bH_\ell$ of $\rho_\ell\otimes\psi_\ell^{-1}$ is also connected reductive.
Let $\mu_N\subset\overline\Q_\ell^\times$ be the
group of $N$th root of unity. }
Define the conjugation-invariant
\begin{equation*}\label{Y}
\bY_\ell:=\{g\in\bH_\ell:~\text{some eigenvalue of } g~\text{belongs to~}\mu_N\backslash\{1\}\}.
\end{equation*}
Then $\bY_\ell$ is a proper closed subvariety of $\bH_\ell$ as it does not contain the identity.
Pick a large enough intermediate local field $\Q_\ell\subset L_\ell\subset \overline\Q_\ell$ such that
$\rho_\ell\otimes\psi_\ell^{-1}(\Gal_K)\subset\GL_{mn}(L_\ell)$ 
(so that $\bH_\ell$ defined over $L_\ell$)
and $\bY_\ell$ is defined over $L_\ell$.
The (algebraic) Chebotarev density theorem (\cite{Se81a},\cite[Theorem 3]{Raj98})
implies that 
\begin{equation}\label{avoidY}
\mathcal L_K:=\{v\in\Sigma_K\backslash S':~\rho_\ell\otimes\psi_\ell^{-1}(\Frob_v)\in 
\bH_\ell(L_\ell)\backslash \bY_\ell(L_\ell) \}
\end{equation}
is of natural (hence Dirichlet) density one. 
Therefore, for any $v\in \mathcal L_K$ the polynomial
$$\prod_{i=1}^m(T-\chi_i(\Frob_v))\in E'[T]$$ 
because it is  the greatest common divisor of \eqref{product}
and $\det(\rho_\ell(\Frob_v)-T\cdot\text{Id})$ in  $E'[T]$. We are done.
\end{proof}

Below is a slight modification of \cite[Chap. III.3.2, Proposition]{Se98}.

\begin{prop}[Serre]\label{Serre1}
Suppose $\phi_\ell:\Gal_K^{\mathrm{ab}}\to \GL_n(\overline\Q_\ell)$ 
is an abelian semisimple $\ell$-adic representation
unramified outside a finite $S\subset\Sigma_K$
such that $\phi_\ell^N$ is locally algebraic for some $N\in\N$.
If there is a number field $E'\subset\overline\Q_\ell$
and a Dirichlet density one subset $\mathcal L_K\subset \Sigma_K\backslash S$ such that 
\begin{equation}\label{trace}
\det(\phi_\ell(\Frob_v)-T\cdot \mathrm{Id})\in  E'[T]\hspace{.2in}\text{for all}~ v\in \mathcal L_K,
\end{equation}
then $\phi_\ell$ is locally algebraic and $E'$-rational.
\end{prop}

\begin{proof}
By the same lines as \cite[Chap. III.3.2, Proposition]{Se98}, one obtains local algebraicity of $\phi_\ell$.
Hence, there exists a modulus $\mathfrak m$ of $K$ such that $\phi_\ell$ factors through some $\overline\Q_\ell$-representation of Serre group $\bS_\mathfrak{m}$ (see $\mathsection$\ref{s24}):
$$\phi_\ell:\Gal_K^{\mathrm{ab}}\stackrel{\epsilon_\ell}{\longrightarrow}\bS_\mathfrak{m}(\overline\Q_\ell)
\stackrel{\phi_0}{\longrightarrow}\GL_n(\overline\Q_\ell).$$
Since the trace $\mathrm{Tr}(\phi_\ell(\Frob_v))\in E'$ for all $v\in\mathcal L_K$ by \eqref{trace} and the subset
$$\{\phi_\ell(\Frob_v):~v\in\mathcal L_K\}\subset \phi_\ell(\Gal_K)$$
is dense by $\mathcal L_K$ of Dirichlet density one, the morphism $\phi_0$ is defined over $E'$ \cite[Chap. II.2.4, Proposition 2 and Lemma]{Se98}. Therefore, $\phi_\ell$ is $E'$-rational.
\end{proof}

\begin{cor}\label{potential}
Let $\rho_\ell:\Gal_K\to\GL_n(\overline\Q_\ell)$ be a semisimple $E'$-rational $\ell$-adic representation 
unramified outside a finite $S\subset\Sigma_K$
such that the algebraic monodromy group $\bG_\ell$ is connected. 
{If $\psi_\ell$ is a weak abelian direct summand of $\rho_\ell$ and there exists a finite field extension $L/K$ such that $\psi_\ell|_{\Gal_L}$ is $E'$-rational,
then $\psi_\ell$ is $E'$-rational.}
\end{cor}

\begin{proof}
Apply Proposition \ref{den1} (by taking $N=[L:K]!$) and Proposition \ref{Serre1}. 
\end{proof}

\subsection{$E$-rationality of abelian part}\label{s26}
Given a semisimple $E$-rational $\ell$-adic representation 
\begin{equation*}
\rho_\ell:\Gal_K\to\GL_n(\overline\Q_\ell),
\end{equation*}
{the \emph{abelian part} of $\rho_\ell$, denoted by $\rho_\ell^{\ab}$,
is the maximal abelian subrepresentation of $\rho_\ell$.
Suppose $\rho_\ell^{\ab}$ is non-zero.
Since $\rho_\ell^{\ab}$ weakly divides $\rho_\ell$, it is locally algebraic by Theorem \ref{abpart}.
Is it also $E$-rational?}
We conjecture the following and give some evidence 
(Theorem \ref{E-rat-evi}).

\begin{conj}\label{E-rat-conj}
If the algebraic monodromy group $\bG_\ell$ of $\rho_\ell$ is connected and
the abelian part $\rho_\ell^{\ab}$ is non-zero,
then $\rho_\ell^{\ab}$ is $E$-rational.
\end{conj}

{Recall that the derived group of the identity component $\bG_\ell^\circ$  is denoted 
by $\bG_\ell^{\der}$};  the group $\bG_\ell^{\der}$ is semismiple.
Decompose the representation $\rho_\ell=\rho_\ell^{\text{c}}\oplus\rho_\ell^{\ab}$
and let $V_\ell^{\text{c}}$ (resp. $V_\ell^{\ab}$) be the representation space of $\rho_\ell^{\text{c}}$ (resp. $\rho_\ell^{\ab}$). Since $\bG_\ell^{\der}$ acts trivially on $V_\ell^{\ab}$,
identify $\bG_\ell^{\der}$ as a subgroup of $\GL_{V_\ell^{\text{c}}}$.

\begin{thm}\label{E-rat-evi}
Let $\rho_\ell$ be a semisimple $E$-rational $\ell$-adic representation such that 
$V_\ell^{\text{c}}$ and $V_\ell^{\ab}$ are both non-zero. 
If $\bG_\ell$ is connected and the formal character of $\bG_\ell^{\der}$ on $V_\ell^{\text{c}}$
does not have  zero as a weight, then $\rho_\ell^{\ab}$ is $E$-rational.
\end{thm}

\begin{proof}
Decompose $\rho_\ell^{\ab}$ as a direct sum of characters:
\begin{equation}\label{rhoD2}
\rho_\ell^{\ab}=\chi_1\oplus\chi_2\oplus\cdots\oplus\chi_m.
\end{equation}
We perform induction on $m\geq 1$.
By Theorem \ref{abpart} and the equivalence in $\mathsection$\ref{s24},
the characters $\chi_1,\chi_2,...,\chi_m$, and 
$\rho_\ell^{\ab}$ are all $E'$-rational for some number field $E'\subset\overline\Q_\ell$, 
which we may assume to be Galois over $E$.
Moreover by \eqref{factor}, $\rho_\ell^{\ab}$ 
is associated to an $E'$-morphism of some Serre group ($\mathfrak m$ a modulus of $K$):
$$\phi_0:  \bS_{\mathfrak{m}}\to \GL_{m,E'}.$$
Let $E_\lambda$ be the closure of $E$ in $\overline\Q_\ell$ (i.e., $\lambda\in\Sigma_E$).
Then $E'\otimes_E E_\lambda=\prod_{\lambda'|\lambda} E'_{\lambda'}$
and by restriction of scalars \cite[Definition 3.4]{BGP19}, we define
\begin{equation}\label{phiE}
\phi_\lambda:\Gal_K^{\ab}\stackrel{\epsilon_\ell}{\rightarrow} \bS_{\mathfrak{m}}(E_\lambda)
\stackrel{\prod_{\lambda'|\lambda}\phi_0\times E'_{\lambda'}}{\longrightarrow}
\GL_n(E'\otimes_E E_\lambda)\subset \GL_{n[E':E]}(E_\lambda)\subset \GL_{n[E':E]}(\overline\Q_\ell),
\end{equation}
which is $E$-rational and contains $\rho_\ell^{\ab}$ as subrepresentation.

Let $L$ be a finite extension of $K$ such that the algebraic monodromy group $\bH_\ell$
of the restriction $(\rho_\ell\oplus \phi_\lambda)|_{\Gal_L}$ is connected. 
Then the Chebotarev density theorem
implies the existence of a density one subset $\mathcal L_L\subset \Sigma_L$ such 
that for any distinct pair of characters $\chi,\chi'$ of $\phi_\lambda|_{\Gal_L}$,
one has 
\begin{equation}\label{distinct}
\chi(\Frob_w)\neq \chi'(\Frob_w)\hspace{.2in}\text{for all}~w\in\mathcal L_L.
\end{equation} 

If $\rho_\ell^{\ab}$ is not $E$-rational, then $\rho_\ell^{\ab}|_{\Gal_L}$ 
is not $E$-rational by Corollary \ref{potential}.
Assume $\rho_\ell^{\ab}|_{\Gal_L}$ does not contain any $E$-rational subrepresentation $\psi_\ell$,
otherwise we can reduce to smaller $m$ (by passing to the quotient mod $\psi_\ell$) and use  the induction hypothesis. 
Suppose 
\begin{equation}\label{kchar}
\chi_1|_{\Gal_L},\chi_2|_{\Gal_L},....,\chi_k|_{\Gal_L}
\end{equation}
is the set of distinct characters in $\rho_\ell^{\ab}|_{\Gal_L}$.
By property \eqref{distinct}, the $k$ numbers 
\begin{equation}\label{knumbers}
\chi_1|_{\Gal_L}(\Frob_w),\chi_2|_{\Gal_L}(\Frob_w),....,\chi_k|_{\Gal_L}(\Frob_w)\in E'
\end{equation}
are distinct for any $w\in\mathcal L_L$.
Since the direct sum of the (locally algebraic) characters in \eqref{kchar} 
is not $E$-rational by assumption,
Proposition \ref{Serre1} implies that  for some $w\in\mathcal L_L$
the distinct numbers in \eqref{knumbers} are not invariant under $\Gal(E'/E)$. 
Hence, the $E$-rational $\phi_\lambda|_{\Gal_L}$ has 
some one-dimensional subrepresentation $\tau$ distinct from \eqref{kchar}.

For almost all $w\in\Sigma_L$, the definition and the $E$-rationality of \eqref{phiE} 
imply that $\tau(\Frob_w)$ must be a $\Gal(E'/E)$-conjugate of 
some eigenvalue of $\rho_\ell^{\ab}|_{\Gal_L}(\Frob_w)$.
As $\tau$ is distinct from \eqref{kchar},
property \eqref{distinct} and $E$-rationality of 
$\rho_\ell|_{\Gal_L}=\rho_\ell^{\text{c}}|_{\Gal_L}\oplus\rho_\ell^{\ab}|_{\Gal_L}$
imply that $\tau(\Frob_w)$ is a root of the characteristic polynomial 
of $\rho_\ell^{\text{c}}|_{\Gal_L}(\Frob_w)$ for all $w\in\mathcal L_L$.

Consider the twisted $\ell$-adic representation $\rho_\ell^{\text{c}}|_{\Gal_L} \otimes \tau^{-1}$
and let $\bG_\tau$ be its algebraic monodromy group.
The faithful representations of (the derived) $\bG_\ell^{\der}$ and $\bG_\tau^{\der}$
on respectively the representation spaces $V_\ell^{\text{c}}$ and $V_\ell^{\text{c}}\otimes\tau^{-1}$
are identified (as twisting by a character will not affect the derived group). 
Thus, the formal character of  $\bG_\tau^{\der}$ 
does not have zero weight (by assumption), 
which implies  that the following conjugation-invariant subvariety of $\bG_\tau$ is proper:
$$\bY_\tau:=\{g\in\bG_\tau:~1~\text{is an eigenvalue of}~g        \}.$$
This is impossible because $\rho_\ell^{\text{c}}|_{\Gal_L} \otimes \tau^{-1} (\Frob_w)\in \bY_\tau$
for all $w\in \mathcal L_L$ and $\mathcal L_L\subset\Sigma_L$ is a density one subset.
\end{proof}

We prove Conjecture \ref{E-rat-conj} when $n=3$.

\begin{cor}\label{conj3}
Let $\rho_\ell$ be a three-dimensional semisimple $E$-rational $\ell$-adic representation 
with non-zero $\rho_\ell^{\ab}$. 
If $\bG_\ell$ is connected, then $\rho_\ell^{\ab}$ is $E$-rational.
\end{cor}

\begin{proof}
We only need to handle the case that $\rho_\ell^{\ab}$ is one-dimensional.
Then $\bG_\ell^{\der}$ is $\SL_2$ on $V_\ell^{\text{c}}$, which has no zero weight 
in the formal character.
We are done by Theorem \ref{E-rat-evi}.
\end{proof}

\subsection{
{Abelian and weak abelian part of compatible system}}
Let $\{\rho_\lambda:\Gal_K\to\GL_n(\overline E_\lambda)\}$
be a semisimple Serre compatible system of $K$ defined over $E$
unramified outside a finite $S\subset\Sigma_K$.
Let $\bG_\lambda$ be the algebraic monodromy group of $\rho_\lambda$.
The connectedness and formal character of $\bG_\lambda$ 
is independent of $\lambda$;  see \cite{Se81b}.
We say that $\{\rho_\lambda\}$ is \emph{connected}
if $\bG_\lambda$ is connected for some $\lambda$ (hence all $\lambda$).
Define $\{\rho_\lambda^{\ab}\}$ (resp. $\{\rho_\lambda^{\wab}\}$, well-defined 
if $\{\rho_\lambda\}$ is connected by Proposition \ref{lem:WeakDiv2} 
and Corollary \ref{cor:WeakDiv2}(i)) as the \emph{abelian part} 
(resp. \emph{weak abelian part}) of the Serre compatible system $\{\rho_\lambda\}$.
One can ask if $\{\rho_\ell^{\ab}\}$ (resp. $\{\rho_\lambda^{\wab}\}$) is still a Serre compatible system defined over $E$ when it is non-zero.
This question  is studied in \cite{Hu20} (in case $E=\Q$) for the abelian part
by using some $\lambda$-independence results of $\bG_\lambda$, such as the following.

\begin{thm}\label{Hui1}\cite[Theorem 3.19, Remark 3.22]{Hu13}
The formal bi-character of $\bG_\lambda$ is independent of $\lambda$.
In particular, the formal character of $\bG_\lambda$ (resp. $\bG_\lambda^{\der}$) 
is independent of~$\lambda$.
\end{thm}

By Theorem \ref{Hui1},
if $\bG_{\lambda_0}^{\der}=\SL_n$ (resp. $\bG_{\lambda_0}=\GL_n$) for some $\lambda_0$ then 
$\bG_\lambda^{\der}=\SL_n$ (resp. $\bG_{\lambda}=\GL_n$) for all $\lambda$.
Thus, we obtain $\lambda$-independence of $\bG_\lambda^{\der}$ (resp. $\bG_{\lambda}$).
However when $n=3$ and $\bG_{\lambda_0}^{\der}=\SL_2\times\{1\}\subset\GL_3$ for some $\lambda_0$, 
the semisimple
$\bG_\lambda^{\der}$ is either  $\SL_2\times\{1\}$ or $\SO_3$  by 
Theorem \ref{Hui1} and we need extra information to exclude 
the second possibility.

\begin{prop}\label{n3}
Let $\{\rho_\lambda\}$ be a connected semisimple three-dimensional Serre compatible system 
of $K$ defined over $E$. Suppose 
$\rho_{\lambda_0}=\sigma_{\lambda_0}\oplus \tau_{\lambda_0}$
is a decomposition into irreducible factors for some $\lambda_0$ 
such that $\sigma_{\lambda_0}$ is two-dimensional. Then the following assertions hold.
\begin{enumerate}[(i)]
\item $\tau_{\lambda_0}$ satisfies ($E$-SCS), i.e., $\tau_{\lambda_0}$ 
belongs to a SCS $\{\tau_\lambda\}$ defined over $E$.
\item If $\rho_\lambda$ is reducible, then $\bG_\lambda^{\der}=\SL_2\times\{1\}$
and the (unique) one-dimensional subrepresentation is $\tau_\lambda$.
\item If $\rho_\lambda$ is irreducible, then $\bG_\lambda^{\der}=\SO_3$ and $\det(\rho_\lambda)=\tau_\lambda^3$.
\end{enumerate}
\end{prop}

\begin{proof}
Assertion (i) follows directly from Corollary \ref{conj3} and $\mathsection$\ref{s24}.
Since $\bG_{\lambda_0}^{\der}=\SL_2\times\{1\}$, 
Theorem \ref{Hui1} implies that $\bG_\lambda^{\der}$ is either $\SL_2\times\{1\}$ or $\SO_3$,
depending on the reducibility of $\rho_\lambda$.
Consider the SCS $\{\rho_\lambda\otimes\tau_\lambda^{-1}\}$ and let $\bG_\lambda'$
be the algebraic monodromy group at $\lambda$. 
Since 
$$\SL_2\times\{1\}\subset \bG_{\lambda_0}'\subset \GL_2\times\{1\},$$ 
the formal character of $\bG_\lambda'$ has a zero weight for all $\lambda$ by Theorem \ref{Hui1}.
Hence, either $\SL_2\times\{1\}\subset \bG_{\lambda}'\subset \GL_2\times\{1\}$
or 
$\bG_\lambda'=\SO_3$,
depending on the reducibility of $\rho_\lambda$.
This implies the last parts in assertions (ii) and (iii).
\end{proof}


\begin{cor}
Let $\{\rho_\lambda\}$ be a connected semisimple three-dimensional Serre compatible system 
of $K$ defined over $E$. If $\rho_\lambda$ is reducible for all $\lambda$,
then $\{\rho_\lambda^{\ab}\}$ is a Serre compatible system defined over $E$.
\end{cor}

\begin{proof}
Let $\lambda_0$ be a finite place of $E$.
If $\rho_{\lambda_0}$ is abelian, then the whole system is abelian by $\mathsection2.4$
and thus $\{\rho_\lambda^{\ab}\}=\{\rho_\lambda\}$ is a compatible system.
If $\rho_{\lambda_0}$ is non-abelian, then it is the direct sum of 
a two-dimensional irreducible representation
and a character. Since $\rho_\lambda$ is reducible for all $\lambda$, 
the assertion follows from Proposition \ref{n3}(i),(ii).
\end{proof}

Although the compatibility of the abelian part $\{\rho_\lambda^{\ab}\}$ is not known,
this is not the case for the weak abelian part $\{\rho_\lambda^{\wab}\}$ by the following theorem. 
Since the formal character of $\bG_\lambda^{\der}$ is independent of $\lambda$ (Theorem \ref{Hui1}), 
the multiplicity of the weight zero in the formal character is independently of $\lambda$; we denote it by $n_0$.\footnote{The independence of $n_0$ of $\lambda$ can also be deduced from the proof of Theorem~\ref{thm:WeakDiv} by choosing the first $\lambda$ in such a way that its $n_0$ is maximal.}




\begin{thm}\label{thm:WeakDiv}
Suppose $\{\rho_\lambda\}$ is connected and $n_0>0$. 
There exists a finite extension $E'$ of $E$ such that if $\{\rho_\lambda\}$ 
is considered as a Serre compatible system $\{\rho_{\lambda'}\}$ of $K$ defined over $E'$\footnote{We define $\rho_{\lambda'}:=\rho_\lambda$ if $\lambda'\in\Sigma_{E'}$ divides $\lambda\in\Sigma_E$.},
then the weak abelian part $\{\rho_{\lambda'}^{\wab}\}$ is an $n_0$-dimensional Serre compatible system defined over~$E'$.
\end{thm}

\begin{proof}
Let $\lambda_0$ be a finite place of $E$. 
Since the weak abelian part $\rho_{\lambda_0}^{\wab}$ of $\rho_{\lambda_0}$ is locally algebraic (Theorem \ref{abpart}),
it can be extended to an abelian 
semisimple Serre compatible system $\{\psi_{\lambda'}\}$ defined over some finite extension $E'$ of $E$
such that $\rho_{\lambda_0}^{\wab}=\psi_{\lambda_0'}$. For $\lambda'\in\Sigma_{E'}$,
since $\psi_{\lambda'}$ is a degree $n_0$ weak abelian direct summand of $\rho_{\lambda'}$,
it coincides with $\rho_{\lambda'}^{\wab}$ by Proposition \ref{lem:WeakDiv2}. We are done.
\end{proof}

\section{Irreducibility and $\ell$-adic Hodge properties}
Suppose $K$ is a totally real field (or CM) field. Let 
$\pi$ be a regular algebraic cuspidal automorphic 
representation of $\GL_n(\A_K)$. 
For a rational prime $\ell$, denote by $S_\ell$ the set of finite places $v\in\Sigma_K$ above $\ell$.

\subsection{Serre compatible system attached to $\pi$}\label{s31}
By Clozel \cite[Theorem 3.13, Lemmas 3.14, 3.15]{Cl90}, $\pi$ is \emph{C-arithmetic}.
By  \cite[Proposition 5.2.3]{BG14}, $\pi$ is C-arithmetic if and only if $\pi\otimes|\det|^{\frac{1-n}{2}}$ 
is \emph{L-arithmetic}. Thus, there exist a number field $E$ and a finite subset $S\subset\Sigma_K$ containing the ramified primes of $\pi$ such that the Satake parameters of $\pi_v\otimes|\det|_v^{\frac{1-n}{2}}$
are defined over $E$ for all $v\in\Sigma_K\backslash S$.
Any field isomorphism $\iota: \C\to \overline\Q_\ell$ induces 
a finite place $\lambda\in\Sigma_E$ dividing $\ell$, so that $\overline\Q_\ell$
can be identified as $\overline E_\lambda$.
Then the semisimple $\ell$-adic representation \eqref{repn} of $K$, 
\begin{equation}\label{pi-repn}
\rho_{\pi,\iota}:\Gal_K\to\GL_n(\overline\Q_\ell)=\GL_n(\overline E_\lambda),
\end{equation}
is $E$-rational and unramified outside $S$. If $\iota':\C\to\overline\Q_\ell$
is another isomorphism that induces the same $\lambda\in\Sigma_E$, then the semisimple
$\rho_{\pi,\iota}$ and $\rho_{\pi,\iota'}$ are equivalent representation by \eqref{lg}, 
which we denote by $\rho_{\lambda}$.
Hence, the set $\{\rho_{\pi,\iota}:~\iota: \C\to \overline\Q_\ell,~\ell~\text{rational prime}\}$
of Galois representations
is equal to the family of $\ell$-adic representations,
\begin{equation}\label{pi-sys}
\{\rho_{\lambda}:\Gal_K\to\GL_n(\overline E_\lambda):~\lambda\in\Sigma_E\},
\end{equation}
which by \eqref{lg}, is a semisimple Serre compatible system of $K$ defined over $E$ unramified outside $S$.
Let $\bG_\lambda$ be the algebraic monodromy group of $\rho_\lambda$.

\subsection{Irreducibility and algebraic monodromy groups}
Suppose $K$ is totally real and $n=3$.
 We prove  the irreducibility of 
$\rho_{\lambda}$ (Theorem \ref{main}) by an L-function argument in \cite{Ram13} 
and study the $\lambda$-independence of $\bG_\lambda$.

\subsubsection{Proof of Theorem \ref{main}.}
If $\rho_\lambda$ is reducible, we have a decomposition $\rho_\lambda=\sigma_\lambda\oplus\tau_\lambda$ 
where $\sigma_\lambda$ is two-dimensional and $\tau_\lambda$ is a character. Then 
\begin{align}\label{Alt}
\begin{split}
\mathrm{Alt}^2(\rho_\lambda)&=(\sigma_\lambda\otimes\tau_\lambda)\oplus \det(\sigma_\lambda),\\
\tau_\lambda^2\oplus \mathrm{Alt}^2(\rho_\lambda)&=(\rho_\lambda\otimes\tau_\lambda)\oplus \det(\sigma_\lambda),\\
1\oplus (\mathrm{Alt}^2(\rho_\lambda)\otimes\tau_\lambda^{-2})&=(\rho_\lambda\otimes\tau_\lambda^{-1})\oplus (\det(\sigma_\lambda)\cdot\tau_\lambda^{-2}).
\end{split}
\end{align}
Since $\tau_\lambda$ is locally algebraic by Theorem \ref{abpart}, it comes from 
an algebraic Hecke character $\tau$ of $K$. 
Since $\rho_\lambda$ is unramified outside $S'=S\cup S_\ell$,
we obtain from \eqref{Alt} and the local-global compatibility \eqref{lg} 
an equation on partial L-functions
\begin{equation}\label{L-eq}
L^{S'}(1,s)L^{S'}(\mathrm{Alt}^2(\pi)\otimes\psi_1,s)=L^{S'}(\pi\otimes\psi_2,s)L^{S'}(\psi_3,s),
\end{equation}
where $\psi_1,\psi_2,\psi_3$ are some algebraic Hecke characters of $K$ and $\psi_3$ corresponds to the locally algebraic character
$\det(\sigma_\lambda)\cdot\tau_\lambda^{-2}$.

Suppose $\det(\sigma_\lambda)\cdot\tau_\lambda^{-2}$ is non-trivial. On the one hand,
the right-hand side of \eqref{L-eq} at $s=1$ belongs to $\C$
by $\psi_3$ non-trivial and $\pi$ cuspidal. On the other hand, $L^{S'}(1,s)$ 
has a simple pole at $s=1$ and $L^{S'}(\mathrm{Alt}^2(\pi)\otimes\psi_1,s)$
is non-zero at $s=1$ by Shahidi \cite[Theorem 1.1]{Sh97}. We get a contradiction.

Suppose $\det(\sigma_\lambda)\cdot\tau_\lambda^{-2}$ is trivial. Then $\rho_\lambda\otimes\tau_\lambda^{-1}$ is self-dual.
Since $\tau_\lambda$ is locally algebraic, $\pi$ is essentially self-dual by multiplicity one \cite{JS81}.
But in this case, $\rho_\lambda$ is irreducible (see \cite[$\mathsection4.4.3$]{Hu23a}).
We conclude that $\rho_\lambda$ is irreducible.\qed

\begin{remark}
F. Calegari suggested to us an alternative proof of Theorem \ref{main} that avoids the use of $\mathrm{Alt}^2(\rho_\lambda)$. 
The key idea, in the case that $\rho_\lambda$ is reducible, was to add a suitable Hecke character $\chi$ to $\pi$, so that the isobaric sum $\pi\boxplus\chi$ is self-dual up to twist. His suggestion initially used Hodge-theoretic properties of $\rho_\lambda$ that currently do not seem to be in the literature unless $\pi$ is polarizable. 
However  it can also be completed by using the results of the previous section as follows:%
\\[.3em]
{\em Suppose $\rho_\lambda$ is reducible with a decomposition $\rho_\lambda=\sigma_\lambda\oplus\tau_\lambda$ 
where $\sigma_\lambda$ is a two-dimensional representation and $\tau_\lambda$ is a character. As above, by Theorem \ref{abpart}, $\tau_\lambda$ 
comes from an algebraic Hecke character $\tau$ of $K$, and so $\tau_\lambda$ is $E'$-rational for some number field $E'\supset E$. 
Hence $\sigma_\lambda$ is $E'$-rational, and this then also follows 
for $\det\sigma_\lambda$, so that again by Theorem \ref{abpart}, 
also $\det\sigma_\lambda$ comes
from an algebraic Hecke character, say $\mu$, of $K$.

Define $\pi':=(\pi\otimes|\det|^{-1})\boxplus(\tau^{-1}\mu)$. Then one checks that $\rho_\lambda\oplus(\tau_\lambda^{-1}\det\sigma_\lambda)$ has contragradient representation isomorphic to its twist by $(\det\sigma_\lambda)^{-1}$, and it follows from multiplicity one that $(\pi')^\vee\cong\pi'\otimes \mu^{-1}$, i.e., that $\pi'$ is essentially self-dual. 
Since $\pi'$ is of isobaric type (3,1), $\pi$ is also essentially self-dual. 
We conclude by the last lines of the previous proof.\qed}
\end{remark}

\subsubsection{Algebraic monodromy groups}
The faithful representation $\bG_\lambda\hookrightarrow\GL_{3,\overline E_\lambda}$
is irreducible for all $\lambda$ by Theorem \ref{main}.
If $\rho_\lambda$ is \emph{Lie-irreducible}\footnote{The faithful representation 
$\bG_\lambda^\circ\to \GL_{3,\overline E_\lambda}$
is irreducible.}, there are two types:
\begin{enumerate}
\item[($A_1$):] $\SO_{3,\overline E_\lambda}\subset \bG_\lambda\subset \GO_{3,\overline E_\lambda}$;
\item[($A_2$):] $\SL_{3,\overline E_\lambda}\subset \bG_\lambda$.
\end{enumerate}
If $\rho_\lambda$ is not Lie-irreducible, there are two types by \cite[Theorems 1,2]{Cl37} (see also \cite[Lemma 4.3]{CG13}):
\begin{enumerate}
\item[(D):] $\rho_\lambda(\Gal_L)$ belongs to the center of $\GL_{3,\overline E_\lambda}$ 
for some finite extension $L/K$;
\item[(Ind):] $\rho_\lambda=\mathrm{Ind}^K_L \psi_\lambda$ is induced from 
a locally algebraic character $\psi_\lambda:\Gal_L\to\overline E_\lambda^\times$ 
for some cubic extension $L/K$ and the three characters of the abelian 
$\rho_\lambda|_{\Gal_{L'}}$ are distinct for some $L'/L$.
\end{enumerate}
Note that $\bG_\lambda^{\der}$ is trivial for type (D) and (Ind).

\begin{thm}\label{mthm}
The irreducible type of $\rho_\lambda$ is independent of $\lambda$. 
Moreover, the following assertion hold.
\begin{enumerate}[(i)]
\item Suppose $\rho_\lambda$ is of type ($A_1$). The automorphic representation $\pi$ is essentially self-dual.
\item Suppose $\rho_\lambda$ is of type (Ind). After replacing $E$ by a finite extension,
the compatible system $\{\rho_\lambda\}_\lambda$ of $K$
is induced from a one-dimensional  $E$-rational compatible system 
$\{\psi_\lambda\}_\lambda$ of a cubic extension $L/K$.
\end{enumerate}
\end{thm}

\begin{proof}
Let $\lambda_0\in\Sigma_E$. 
If $\rho_{\lambda_0}$ is of type ($A_1$) (resp. ($A_2$)),
then $\bG_{\lambda_0}^{\der}=\SO_3$ (resp. $\SL_3$).
Irreducibility of $\rho_\lambda$ and $\lambda$-independence of the formal character of $\bG_\lambda^{\der}$ 
(Theorem \ref{Hui1}) imply that $\bG_\lambda^{\der}=\SO_3$ (resp. $\SL_3$).
Moreover, the $E$-rational adjoint representation 
$\ad(\rho_{\lambda_0}):=\rho_{\lambda_0}\otimes\rho_{\lambda_0}^\vee$
is the sum of an irreducible eight-dimensional representation and a locally algebraic 
character $\tau_{\lambda_0}=\ad(\rho_{\lambda_0})^{\ab}$ (Theorem \ref{abpart})
such that $\rho_{\lambda_0}\cong  \rho_{\lambda_0}^\vee\otimes \tau_{\lambda_0}.$
Hence, $\pi$ is essentially self-dual by multiplicity one \cite{JS81}.

If $\rho_{\lambda_0}$ is of type (D), the adjoint representation 
$\ad(\rho_{\lambda_0})=\rho_{\lambda_0}\otimes\rho_{\lambda_0}^\vee$ has finite image.
Since $\{\ad(\rho_\lambda):=\rho_{\lambda}\otimes\rho_{\lambda}^\vee\}$ is a Serre compatible system,
 $\ad(\rho_{\lambda})$ has finite image  
for all $\lambda$ by \cite{Se81b}, which implies $\rho_\lambda(\Gal_L)$ belong to the homothety for all $\lambda$.

If  $\rho_{\lambda_0}$ is of type (Ind), then
it is potentially abelian and $E$-rational, which 
imply that $\psi_{\lambda_0}$ is $E'$-rational for 
some $E\subset E'\subset\overline E_\lambda$.
There exists a Serre compatible system $\{\psi_{\lambda'}:~\lambda'\in\Sigma_{E'}\}$ 
of $L$ defined over $E'$ extending $\psi_{\lambda_0}$
and moreover, $\{\mathrm{Ind}^L_K \psi_{\lambda'}\}$ is a semisimple Serre compatible system of $K$
containing $\{\rho_\lambda\}$.
Since $\rho_{\lambda_0}|_{\Gal_{L'}}$ is a direct sum of three distinct locally algebraic characters,
this is true for $\rho_\lambda|_{\Gal_{L'}}$ for all $\lambda$ by compatibility.
We are done.
\end{proof}

\begin{remark}
The arguments in this section do not use the regularity of $\pi$.
In next subsection, this condition will be used to prove that $\rho_\lambda$ 
cannot be of type (D), see Proposition \ref{noD}.
\end{remark}

\subsection{$\ell$-adic Hodge properties}
Suppose $K$ is totally real and $\pi$ is a regular algebraic 
cuspidal automorphic representation of $\GL_3(\A_K)$. 
The \emph{weight}  $(a_{\sigma,i})\in (\Z^3)^{\text{Hom}(K,\C)}$ of $\pi$
(see \cite[$\mathsection2.1$]{BLGGT14b}) satisfies
$$a_{\sigma,1}\geq a_{\sigma,2}\geq a_{\sigma,3}$$
for all $\sigma\in \text{Hom}(K,\C)$.
Let $\{\rho_\lambda:=\rho_{\pi,\iota_\lambda}\}$ be the SCS defined over $E$ attached to $\pi$ 
in \eqref{pi-sys}, where 
$\iota_\lambda:\C\to\overline E_\lambda$ is an isomorphism for all $\lambda$.
Recall that  the residue characteristic of $\lambda\in\Sigma_E$ is denoted $\ell(\lambda)$.

By Theorem \ref{mthm}, we define the irreducible type of the SCS $\{\rho_\lambda\}$ 
as the type of $\rho_\lambda$.
If $\{\rho_\lambda\}$ is of type (D) or (Ind), then the SCS is potentially abelian 
and thus $\rho_\lambda$ is de Rham for all $\lambda$.
If $\{\rho_\lambda\}$ is of type ($A_1$), then $\pi$ is essentially self-dual (Theorem \ref{mthm}(i))
and by \cite{BLGGT12,BLGGT14a} and \cite{Ca14},  
the following $\ell$-adic Hodge properties hold for $v\in S_{\ell(\lambda)}$.
\begin{enumerate}
\item[(dR):]  The local representation $\rho_\lambda|_{\Gal_{K_v}}$ is de Rham
with (distinct) $\tau$-Hodge-Tate numbers 
$$HT_\tau:=\{a_{\iota_\lambda^{-1}\circ\tau}+2,~a_{\iota_\lambda^{-1}\circ\tau}+1,~a_{\iota_\lambda^{-1}\circ\tau}\},$$
where $\tau:K_v\to \overline E_\lambda$.
\item[(Cr):] If $\pi_v$ is unramified, then $\rho_\lambda|_{\Gal_{K_v}}$ is crystalline.
\end{enumerate}

\noindent Denote  by $\F_\lambda$ the residue field of $E_\lambda$ and
by $\bar\rho_\lambda:\Gal_K\to\GL_3(\overline\F_\lambda)$ the residual representation
given by the \emph{semisimplified reduction} of $\rho_\lambda$.
A recent result of A'Campo \cite[Theorem 1.0.6]{A'C23} asserts that 
the $\ell$-adic Hodge properties (dR) and (Cr) hold if 
\begin{enumerate}
\item[(RI):] $\rho_\lambda$ is residually irreducible (i.e., $\bar\rho_\lambda$ is irreducible) and
\item[(DG):] $\bar\rho_\lambda$ is \emph{decomposed generic} (see \cite[Definition 4.3.1]{ACC+23}).
\end{enumerate}
A sufficient condition of (DG) is as follows.

\begin{lemma}\label{lem3}\cite[Lemma 7.1.5]{ACC+23}
Suppose $L$ is a number field and $\bar r:\Gal_L\to \GL_n(\overline\F_\ell)$ is a continuous representation.
Denote by $\ad(\bar r)$ the adjoint representation $\bar r\otimes\bar r^\vee$ and by
$M$ the normal closure of $\overline L^{\mathrm{Ker}(\ad(\bar r))}$ over $\Q$.
If $M$ does not contain a primitive $\ell$th root of unity, then $\bar r$ is decomposed generic.
\end{lemma}

\begin{remark}
Conversely, if $\{\rho_\lambda\}$ is a strictly compatible system in the sense of \cite[Section 5.1]{BLGGT14b} 
with (distinct) $\tau$-Hodge-Tate weights (e.g., when $\{\rho_\lambda\}$ satisfies the local-global compatibility \eqref{lg} for all $v\in\Sigma_K$ and also (dR) and (Cr) above), then (RI) holds for 
almost all $\lambda$ by Theorem \ref{main}
and \cite[Theorem 1.1(iii)]{Hu23b}.
\end{remark}

We are now ready to exclude irreducible type (D).

\begin{prop}\label{noD}
The compatible system $\{\rho_\lambda\}$ cannot be of type (D).
\end{prop}

\begin{proof}
Suppose $\{\rho_\lambda\}$ is of type (D). 
Consider the SCS $\{\ad(\rho_\lambda):=\rho_\lambda\otimes\rho_\lambda^\vee\}$. 
By Serre \cite{Se81b}, the finite $\mathrm{Ker}(\ad(\rho_\lambda))$ is independent of $\lambda$
and is equal to $\Gal(K'/K)$.
Let $M$ be the normal closure of $K'$ over $\Q$ and $m:=[M:\Q]$.

Pick $\lambda\in\Sigma_E$ such that 
$\ell:=\ell(\lambda)>1+m$. 
Since the normal closure of $\overline K^{\mathrm{Ker}(\ad(\bar \rho_\lambda))}$ over $\Q$
is also equal to $M$ and $M$ cannot contain a primitive $\ell$th root 
of unity,
 Lemma \ref{lem3} implies that $\bar\rho_\lambda$ is 
decomposed generic (DG). 
Moreover as $\ell>m\geq |\ad(\bar\rho_\lambda)|$, \cite[Proposition 43(ii)]{Se77} 
asserts that 
the irreducible factors of the semisimple representations $\ad(\rho_\lambda)$ and $\ad(\bar\rho_\lambda)$ 
correspond via reduction.
Since the multiplicity of the trivial representation of $\ad(\rho_\lambda)$ 
is one (by the irreducibility of $\rho_\lambda$), this also holds for $\ad(\bar\rho_\lambda)$,
which implies that $\bar\rho_\lambda$ is irreducible (RI).
Therefore, \cite[Theorem 1.0.6]{A'C23} asserts that $\rho_\lambda|_{\Gal_{K_v}}$ (for $v\in S_\ell$)
is de Rham with
distinct $\tau$-Hodge-Tate numbers.
But this contradicts that $\rho_\lambda(\Gal_L)$ belongs to the homothety for some finite extension $L/K$. 
\end{proof}

It remains to study the $\ell$-adic Hodge properties of $\{\rho_\lambda\}$ 
when it is of type ($A_2$).

\begin{thm}\label{mthm2}
Suppose $\{\rho_\lambda\}$ is of type ($A_2$).
Then there exists a Dirichlet density one set $\mathcal L\subset\Sigma_\Q$ of primes $\ell$
such that $\rho_\lambda$ satisfies the above properties (dR) and (Cr) if $\ell\in\mathcal{L}$.
\end{thm}

\begin{proof}
By \cite[Theorem 1.0.6]{A'C23} and Lemma \ref{lem3}, it suffices to show that for a Dirichlet density one set of primes $\ell$, the following two assertions hold whenever $\ell(\lambda)=\ell$:
\begin{enumerate}
\item[(RI):] the residual representation $\bar\rho_\lambda$ is irreducible and  
\end{enumerate}
\begin{equation}\label{noroot}
\Q(\zeta_\ell)\nsubseteq M_\lambda
\end{equation}
where $\zeta_\ell$ is a primitive $\ell$th root of unity and $M_\lambda$ 
is the normal closure of $\overline K^{\mathrm{Ker}(\ad(\bar\rho_\lambda))}$ over $\Q$.

Since $\{\rho_\lambda\}$ is of type ($A_2$) and 
the validity of (RI) (resp. \eqref{noroot}) is invariant under
twisting the SCS with some power of the $\ell$-adic cyclotomic system,
we assume that the algebraic monodromy group $\bG_\lambda=\GL_{3,\overline E_\lambda}$ 
for all $\lambda$ (Theorem \ref{Hui1}).
Fix $\lambda'\in\Sigma_E$, \eqref{eq:steinberg} and the Chebotarev density theorem 
imply that $\rho_{\lambda'}(\Frob_v)$ is regular semisimple (i.e., has three distinct eigenvalues) 
in $\bG_{\lambda'}=\GL_{3,\overline E_{\lambda'}}$
for a density one set $\mathcal L_K\subset\Sigma_K$ of $v$. 
We pick two distinct $v_1,v_2\in\mathcal L_K\backslash S$.
By enlarging the coefficient field $E$ in $\mathsection$\ref{s31},
we assume $E$ contains the eigenvalues of $\rho_{\lambda'}(\Frob_{v_1})$ and $\rho_{\lambda'}(\Frob_{v_2})$.
For any $\lambda\in\Sigma_E$, the compatibility condition of $\{\rho_\lambda\}$ and the enlarged field $E$ 
imply that 
\begin{itemize}
\item $\{\mathrm{Tr}(\rho_\lambda(\gamma)):~\gamma\in\Gal_K\}\subset E_\lambda$ and 
\item there exists $\gamma\in\Gal_K$ such that $\rho_\lambda(\gamma)$ has three distinct roots in $E$.
\end{itemize}

\noindent Hence, we may assume $\rho_\lambda(\Gal_K)\subset \GL_3(E_\lambda)$
for all $\lambda$ by \cite[Lemma A.1.5]{BLGGT14b}.

Consider the following semisimple Serre compatible systems (defined over $\Q$ unramified outside $S$) given by  restriction of scalars:
\begin{equation}\label{res}
\{\Phi_\ell:\Gal_K\stackrel{\prod_{\lambda|\ell}\rho_\lambda}{\longrightarrow} 
\prod_{\lambda|\ell}\GL_3(E_\lambda)
=\GL_3(E\otimes\Q_\ell)=\mathrm{Res}_{E/\Q}(\GL_{3,E})(\Q_\ell)\subset\GL_{3[E:\Q]}(\Q_\ell)\}_\ell
\end{equation}

and 

\begin{equation}\label{resad}
\{\Phi_\ell^{\ad}:\Gal_K\stackrel{\prod_{\lambda|\ell}\ad(\rho_\lambda)}{\longrightarrow} 
\prod_{\lambda|\ell}\GL_9(E_\lambda)
=\GL_9(E\otimes\Q_\ell)=\mathrm{Res}_{E/\Q}(\GL_{9,E})(\Q_\ell)\subset\GL_{9[E:\Q]}(\Q_\ell)\}_\ell.
\end{equation}

Let $\bH_\ell$ be the algebraic monodromy group of $\Phi_\ell$.
By Serre \cite{Se81b}, one can pick a finite extension $L$ of $K$ that is Galois over $\Q$ such that 
$\Phi_\ell(\Gal_L)\subset\bH_\ell^\circ$ for all $\ell$.
It is enough to show that for a density one set $\mathcal L$ of rational primes $\ell$, 
the following two assertions hold whenever $\ell(\lambda)=\ell$:
\begin{enumerate}
\item[(RI$_L$):] the residual representation $\bar\rho_\lambda|_{\Gal_L}$ is (absolutely) irreducible and  
\end{enumerate}
\begin{equation}\label{norootL}
\Q(\zeta_\ell)\nsubseteq \Omega_\ell, 
\end{equation}
where $\Omega_\ell$ is the normal closure of 
$\overline L^{\mathrm{Ker}(\Phi_\ell^{\ad}|_{\Gal_L})}$ over $\Q$. Note that $\Omega_\ell/\Q$ is an infinite extension. 

\subsubsection{Proof of (RI$_L$) for density one}\label{RIL}
Let $\bH_\ell^{\ad}$ be the algebraic monodromy group of $\Phi_\ell^{\ad}|_{\Gal_L}$.
Since $\Phi_\ell^{\ad}|_{\Gal_L}$ factors through $\Phi_\ell|_{\Gal_L}$ for all $\ell$ and 
the algebraic monodromy group of $\ad(\rho_\lambda)$ is $\PGL_{3}$ for all $\lambda$,
 $\bH_\ell^{\ad}$ is isomorphic to the adjoint quotient of $\bH_\ell^\circ$, which is connected and adjoint semisimple. 

Let $\alpha_\ell:\bH_\ell^{\sc}\to \bH_\ell^{\ad}$ be the universal covering. By Larsen \cite[Theorem 3.17]{La95},
there is a Dirichlet density one set $\mathcal{L}'$ of primes $\ell$ such that 
the preimage $\Gamma_\ell^{\sc}$ of $\Gamma_\ell^{\ad}:=\Phi_\ell^{\ad}(\Gal_L)$ 
in $\bH_\ell^{\sc}(\Q_\ell)$ is a hyperspecial maximal compact subgroup
for all $\ell\in\mathcal L'$. It follows from \cite[$\mathsection4.2$ Proof of Theorem 1.2]{HL20}
that if $\ell\in\mathcal L'$ is sufficiently large, the following equality on invariant dimensions hold:
\begin{equation}\label{invdim}
\dim_{\Q_\ell} \ad(\Phi_\ell^{\ad})^{\Gal_L}=\dim_{\F_\ell}\ad(\bar\Phi_\ell^{\ad})^{\Gal_L},
\end{equation}
but this implies 
\begin{equation}\label{invdim2}
\dim_{\Q_\ell} (\Phi_\ell^{\ad})^{\Gal_L}=\dim_{\F_\ell}(\bar\Phi_\ell^{\ad})^{\Gal_L}
\end{equation}
since $\Phi_\ell^{\ad}$ is a subrepresentation of $\ad(\Phi_\ell^{\ad})$.
Here $\bar\Phi_\ell^{\ad}$ is the (semisimple) residual representation of $\Phi_\ell^{\ad}$
and $\ad(\bar\Phi_\ell^{\ad})$ (resp. $\ad(\bar\Phi_\ell^{\ad})|_{\Gal_L}$ as $L/K$ Galois) 
is semisimple \cite{Se94}.

As $\{\rho_\lambda\}$ is of type ($A_2$), $\rho_\lambda|_{\Gal_L}$
is absolutely irreducible for all $\lambda$. 
By \eqref{res} and \eqref{invdim2}, $\bar\rho_\lambda|_{\Gal_L}$ is 
absolutely irreducible if $\ell$ belongs to a Dirichlet density one set $\mathcal L''$.

\subsubsection{Proof of \eqref{norootL} for density one}
We need some preparations before the proof.

\begin{defi}
Let $G$ be a finite group and denote by $\mathrm{JH}(G)$ the multiset of Jordan-H\"older factors of $G$.
Define $c_\ell(G)\in\N$ to be the product of the orders
of the cyclic elements of $\mathrm{JH}(G)$ not isomorphic to $\Z/\ell\Z$.
\end{defi}

\begin{lemma}\label{cf}
Let $H_1,H_2,...,H_k$ be finite groups and $G$ a subgroup of $H_1\times H_2\times\cdots \times H_k$.
If $G$ surjects onto $H_i$ for all $1\leq i\leq k$, then 
$\mathrm{JH}(G)\subset \mathrm{JH}( H_1\times H_2\times\cdots \times H_k)$ as inclusion of multisets.
\end{lemma}

\begin{proof}
For $1\leq i\leq k$, let $N_i$ be the kernel of $G$ projecting to $H_1\times H_2\times\cdots\times H_i$.
Then we have a chain of normal subgroups (of $G$)
$$\{e\}=N_k\subset \cdots \subset N_2\subset N_1\subset G$$
such that $N_i/N_{i+1}$ injects into $H_{i+1}$. Since $G$ surjects onto $H_{i+1}$ and $N_i$ is normal in $G$,
$N_i/N_{i+1}$ is normal in $H_{i+1}$. We are done.
\end{proof}

\vspace{.1in}

Suppose $\ell\in\mathcal L'$ (defined in $\mathsection$\ref{RIL}) is sufficiently large such that 
$\Gal(L(\zeta_\ell)/L)\simeq\Z/(\ell-1)\Z$. 
We would like to show that $\zeta\notin \Omega_\ell$. 
If  $\zeta_\ell\in\Omega_\ell$ on the contrary, we obtain a surjective homomorphism
$$\psi_\ell: \Gal(\Omega_\ell/L)\to  \Z/(\ell-1)\Z. $$
 Let $L_1:=\overline L^{\text{Ker} (\Phi_\ell^{\ad}|_{\Gal_L})}$ and denote by
\begin{equation}\label{Li}
L_1,L_2,...,L_k
\end{equation}
be the $\Gal_\Q$-conjugates of $L_1$. Note that
\begin{enumerate}[(i)]
\item $L_1,L_2...,L_k$ generate $\Omega_\ell$;
\item $k\leq [L:\Q]$;
\item $\Gal(L_i/L)\cong\Gamma_\ell^{\ad}:=\Phi_\ell^{\ad}(\Gal_L)$ for all $1\leq i\leq k$. 
\end{enumerate}
By looking at the action of $G_\ell:=\Gal(\Omega_\ell/L)$ on \eqref{Li}, we can identify it
as a subgroup of 
$$\prod_{i=1}^k \Gal(L_i/L)\simeq \prod_{i=1}^k \Gamma_\ell^{\ad}$$ 
that surjects onto the $i$th factor $\Gamma_\ell^{\ad}$
for all $i$. 
Identify $\Gamma_\ell^{\ad}$ as a subgroup of $\GL_{9[E:\Q]}(\Z_\ell)$ using \eqref{resad}
and let $\bar\Gamma_\ell^{\ad}$ be the residual image in $\GL_{9[E:\Q]}(\F_\ell)$.
Let $\bar G_\ell$ be the image of $G_\ell$ in 
\begin{equation}\label{finiteproduct}
\prod_{i=1}^k \bar\Gamma_\ell^{\ad}.
\end{equation} 
Since the kernel of $\GL_{9[E:\Q]}(\Z_\ell)\to \GL_{9[E:\Q]}(\F_\ell)$ is pro-$\ell$,
the surjective $\psi_\ell$ factors through $\bar G_\ell$.
Then Lemma \ref{cf} implies that 
\begin{equation}\label{wrong}
\mathrm{JH}(\Z/(\ell-1)\Z)\subset \mathrm{JH}(\prod_{i=1}^k \bar\Gamma_\ell^{\ad}).
\end{equation}

Consider the central isogeny $\alpha_\ell:\bH_\ell^{\sc}\to \bH_\ell^{\ad}$ (defined in $\mathsection$\ref{RIL}). By \cite[Corollary 2.5]{HL20}, for $\ell\in\mathcal L'$ the image $\alpha_\ell(\Gamma_\ell^{\sc})$
is normal in $\Gamma_\ell^{\ad}$ with index
$$[\Gamma_\ell^{\ad}:\alpha_\ell(\Gamma_\ell^{\sc})]\leq C$$
for some constant $C$ depending only on $9[E:\Q]$. Thus, one obtains 
\begin{equation}\label{Cbound}
[\bar\Gamma_\ell^{\ad}:\overline{\alpha_\ell(\Gamma_\ell^{\sc})}]\leq C
\end{equation}
where $\overline{\alpha_\ell(\Gamma_\ell^{\sc})}$ is the image of $\alpha_\ell(\Gamma_\ell^{\sc})$
in $\bar\Gamma_\ell^{\ad}$. Since $\bar\Gamma_\ell^{\ad}$ is hyperspecial maximal compact,
it is equal to $\cH_{\Z_\ell}(\Z_\ell)$ for some semisimple group scheme $\cH_{\Z_\ell}$ over $\Z_\ell$
whose generic fiber is $\bH_\ell^{\sc}$. The special fiber $\cH_{\F_\ell}$ is connected semisimple over $\F_\ell$
and of rank bounded by $9[E:\Q]$.
Since the kernel of 
$$\Gamma_\ell^{\sc}=\cH_{\Z_\ell}(\Z_\ell)\twoheadrightarrow \cH_{\F_\ell}(\F_\ell)$$
is pro-$\ell$ and $c_\ell(\cH_{\F_\ell}(\F_\ell))$ is bounded by a constant $C'$ 
depending only on $9[E:\Q]$. As $\Gamma_\ell^{\sc}$ surjects onto $\overline{\alpha_\ell(\Gamma_\ell^{\sc})}$, one obtains 
\begin{equation}\label{C'bound}
c_\ell(\overline{\alpha_\ell(\Gamma_\ell^{\sc})})\leq C'.
\end{equation}
Hence, \eqref{Cbound}, \eqref{C'bound}, and (ii) imply that 
\begin{equation}\label{kbound} 
c_\ell(\prod_{i=1}^k \bar\Gamma_\ell^{\ad})\leq (CC')^k \leq (CC')^{[L:\Q]}.
\end{equation}
Therefore, \eqref{wrong} cannot hold (and thus $\zeta_\ell\notin \Omega_\ell$) if $\ell-1>(CC')^{[L:\Q]}$. 
It follows that \eqref{norootL} holds for all $\ell$ belonging to some density one set $\mathcal L'''$.

We are done by taking $\mathcal L$ to be the intersection $\mathcal L''\cap\mathcal L'''$.
\end{proof}

\section*{Acknowledgments}
G.~B\"ockle and C.-Y. Hui would like to thank F. Calegari, T. Gee, C. Khare, M. Larsen, and J. Thorne for  feedback and comments on this work.
 G.~B\"ockle acknowledges support by the Deutsche Forschungsgemeinschaft (DFG, German Research Foundation) through the Collaborative Research Centre TRR 326 Geometry and Arithmetic of Uniformized Structures, project number 444845124. C.-Y. Hui was partially supported by NSFC (no. 12222120) and a Humboldt Research Fellowship.

\vspace{.1in}
\end{document}